\documentclass[a4paper,10pt,leqno]{amsart}
\title[The $L^2$-torsion function and the Thurston norm of $3$-manifolds]
{The $L^2$-torsion function and the Thurston norm of $3$-manifolds}

\author{Stefan Friedl}
\address{Fakult\"at f\"ur Mathematik\\ Universit\"at Regensburg\\93040 Regensburg\\   Germany}
\email{sfriedl@gmail.com}
\urladdr{http://www.uni-regensburg.de/Fakultaeten/nat\_Fak\_I/friedl/index.html}
\author{Wolfgang L\"uck}
        \address{Mathematisches Institut der Universit\"at Bonn\\
                Endenicher Allee 60\\
                53115 Bonn, Germany}
         \email{wolfgang.lueck@him.uni-bonn.de}
          \urladdr{http://www.him.uni-bonn.de/lueck}
        \date{November 2015}
\keywords{$L^2$-Betti numbers, $L^2$-torsion, twisting with finite-dimensional representations, Thurston norm}
    \subjclass[2010]{57M27, 57Q10,  58J52, 22D25}

\usepackage{hyperref}
\usepackage{color}
\usepackage{pdfsync}
\usepackage{calc}
\usepackage{enumerate,amssymb}
\usepackage[arrow,curve,matrix,tips,2cell]{xy}
  \SelectTips{eu}{10} \UseTips
  \UseAllTwocells

\DeclareMathAlphabet\EuR{U}{eur}{m}{n}
\SetMathAlphabet\EuR{bold}{U}{eur}{b}{n}

\makeindex             

\theoremstyle{plain}
\newtheorem{theorem}{Theorem}[section]
\newtheorem{lemma}[theorem]{Lemma}

\theoremstyle{definition}

\newtheorem{definition}[theorem]{Definition}
\newtheorem{example}[theorem]{Example}

\newtheorem{remark}[theorem]{Remark}

\newtheorem{convention}[theorem]{Convention}
\newtheorem*{claim}{Claim}

{\catcode`@=11\global\let\c@equation=\c@theorem}

\newcommand{\comsquare}[8]                   
{\begin{CD}
#1 @>#2>> #3\\
@V{#4}VV @V{#5}VV\\
#6 @>#7>> #8
\end{CD}
}

\newcommand{\xycomsquare}[8]                   
{\xymatrix
{#1 \ar[r]^{#2} \ar[d]^{#4} &
#3 \ar[d]^{#5}  \\
#6\ar[r]^{#7} &
#8
}
}

\newcommand{\xycomsquareminus}[8]                      
{\xymatrix{#1 \ar[r]^-{#2} \ar[d]^-{#4} &
#3 \ar[d]^-{#5}  \\
#6\ar[r]^-{#7} &
#8
}
}

\newcommand{\caln}{{\mathcal N}}

\newcommand{\IC}{{\mathbb C}}

\newcommand{\IN}{{\mathbb N}}

\newcommand{\IQ}{{\mathbb Q}}
\newcommand{\IR}{{\mathbb R}}

\newcommand{\IZ}{{\mathbb Z}}

\newcommand{\eul}{\operatorname{Eul}}

\newcommand{\id}{\operatorname{id}}

\newcommand{\im}{\operatorname{im}}

\newcommand{\pr}{\operatorname{pr}}

\newcommand{\supp}{\operatorname{supp}}

\newcommand{\tors}{\operatorname{tors}}

\newcommand{\spinc}{\operatorname{Spin}^c}

\newcommand{\largehomo}{$(H_1)_f$-factorizing} 
\newcommand{\largehomoblank}{\largehomo\ }

\newcommand{\tmfrac}[2]{\mbox{\large$\frac{#1}{#2}$}}

\newcommand{\higherlim}[3]{{\setbox1=\hbox{\rm lim}
        \setbox2=\hbox to \wd1{\leftarrowfill} \ht2=0pt \dp2=-1pt
        \mathop{\vtop{\baselineskip=5pt\box1\box2}}
        _{#1}}^{#2}#3}

\newcommand{\version}[1]                       
{\begin{center} last edited on #1\\
last compiled on \today\\
name of texfile: \jobname
\end{center}
}

\newcounter{commentcounter}

\begin{document}

\typeout{----------------------------  twist.tex  ----------------------------}


\typeout{------------------------------------ Abstract ----------------------------------------}

\begin{abstract}
  Let $M$ be an oriented
  irreducible $3$-manifold with infinite fundamental group and empty
  or toroidal boundary which is not $S^1 \times D^2$.  Consider any element $\phi$ in the first cohomology of $M$
  with integer coefficients.  Then one can define the $\phi$-twisted $L^2$-torsion function of the
  universal covering which is a function from the set of positive real numbers to the set of real numbers.
 By earlier work of the second author and Schick the evaluation at $t=1$ determines the volume.

  In this paper we show  that its  degree, which is a number extracted from its asymptotic
  behavior at $0$ and at $\infty$, agrees with the Thurston norm of $\phi$. 
\end{abstract}

\maketitle


 \typeout{-------------------------------   Section 0: Introduction --------------------------------}

\setcounter{section}{-1}
\section{Introduction}
Reidemeister torsion is one of  the first invariants in algebraic topology which are
able to distinguish the diffeomorphism type of closed manifolds which are homotopy equivalent. A
prominent example is the complete classification of lens spaces, see for
instance~\cite{Cohen(1973)}.  The Alexander polynomial, which is one of the basic
invariants for knots and $3$-manifolds, can be interpreted as Reidemeister torsion, see for
instance~\cite{Milnor(1968a)}.  The Reidemeister torsion of a 3-manifold can be generalized in two ways.  Either one
can twist it with an element in the first cohomology which leads for example  to the
twisted Alexander polynomial, see for instance~\cite{Friedl-Vidussi(2011survey)}, or one
can consider the $L^2$-version of appropriate coverings resulting in
$L^2$-torsion invariants, see for instance~\cite[Chapter~3]{Lueck(2002)}.  Recently there
have been attempts to combine these two generalizations and consider twisted
$L^2$-versions. Such generalizations have been considered under the name of
\emph{$L^2$-Alexander torsion} or \emph{$L^2$-Alexander Conway} invariants for knots or $3$-manifolds,
for instance
in~\cite{Dubois-Friedl-Lueck(2014Alexander), Dubois-Friedl-Lueck(2014symmetric), Dubois-Friedl-Lueck(2015flavors),
  Dubois-Wegner(2010), Dubois-Wegner(2015), Li-Zhang(2006volume),
  Li-Zhang(2006Alexander),Li-Zhang(2008)}.

In all of these papers one has to make certain assumptions to ensure that the twisted
$L^2$-versions are well-defined. They concern $L^2$-acyclicity and determinant class. Either
these conditions were just assumed to hold, or verified in special cases by a direct
computation. A systematic study of these twisted $L^2$-invariants under the name 
\emph{$L^2$-torsion function} has been carried out in~\cite{Lueck(2015twisting)}. 
We summarize some of the results of~\cite{Lueck(2015twisting)} for 3-manifolds.
Let $M$ be  a 3-manifold. (Here and throughout the paper we assume that all 3-manifolds are 
compact, connected and oriented with empty or toroidal boundary, unless we say explicitly otherwise.)
If $M$ is irreducible and if it has infinite fundamental group,  then it was shown in~\cite{Lueck(2015twisting)}  that 
all these necessary conditions are satisfied for the universal covering $\widetilde{M}$ and an 
element $\phi \in H^1(M;\IZ)$. The result is an equivalence class of functions
\[
\overline{\rho}^{(2)}({M};\phi) \colon (0,\infty) \to \IR
\]
where we call two functions $f, g \colon (0,\infty) \to \IR$ equivalent if for some integer $m$ we have
$f(t) - g(t) = m \cdot \ln(t)$. We recall the definition in Section~\ref{subsec:The_phi-twisted_L2-torsion_function}.
Note though that this invariant is \emph{minus the logarithm} of the function defined and studied in the aforementioned papers.
In those papers the corresponding function  was usually referred to as the $L^2$-Alexander torsion. 
The convention of this paper brings us in line with~\cite{Lueck(2002)}.

 The evaluation of $\overline{\rho}^{(2)}({M};\phi)$ at $t=1$ is well-defined  and in fact it was shown 
by the second author and Schick~\cite[Theorem~0.7]{Lueck-Schick(1999)} that for any
irreducible 3-manifold  we have
\[ 
\overline{\rho}^{(2)}({M};\phi)(t=1)=-\frac{1}{6\pi} \operatorname{vol}(N),\]
where $\operatorname{vol}(N)$ equals the sum of the volumes of the hyperbolic pieces in the JSJ-decomposition of $N$. 

In the sequence of papers \cite{Dubois-Friedl-Lueck(2014Alexander)},
\cite{Dubois-Friedl-Lueck(2014symmetric)}, \cite{Dubois-Friedl-Lueck(2015flavors)} and \cite{Lueck(2015twisting)} the behavior of 
$\overline{\rho}^{(2)}({M};\phi)$ as $t$ `goes to the extremes', i.e.\ as $t\to 0$ and $t\to \infty$, was studied. 
In particular in~\cite{Lueck(2015twisting)} it was shown 
that for any representative $\rho$
there exist constants $C \ge  0$ and $D \ge  0$ such that  we get  for  $0 < t \le 1$
\[
C \cdot \ln(t) -D
\le 
\rho(t)
\le
- C \cdot \ln(t) + D,
\]
and for  $t \ge 1$
\[
- C \cdot \ln(t)  -D
\le 
\rho(t)
\le
C  \cdot \ln(t) + D.
\]
Hence $\limsup_{t \to 0} \frac{\rho(t))}{\ln(t)}$ and $\liminf_{t \to \infty}
\frac{\rho(t))}{\ln(t)}$  exist and we can define the \emph{degree} of
$\overline{\rho}^{(2)}({M};\phi)$ to be
\[
\deg\bigl(\overline{\rho}^{(2)}({M};\phi)\bigr) := \limsup_{t \to \infty} \frac{\rho(t)}{\ln(t)} - \liminf_{t \to 0} \frac{\rho(t)}{\ln(t)}.
\]
It is obviously independent of the choice of the representative $\rho$. 

Thurston~\cite{Thurston(1986norm)} assigned to $\phi$ another invariant, its
\emph{Thurston norm} $x_M(\phi)$, which we will review in
Subsection~\ref{subsec:Thurston_norm}.

The main result of our paper says that the functions $\overline{\rho}^{(2)}({M};\phi)$ not
only determine the volume of a 3-manifold but that they also determine the Thurston
norm. More precisely, we have the following theorem.

\begin{theorem} \label{the:main_result_introduction} 
  Let $M$ be an irreducible $3$-manifold with infinite fundamental group and empty or
  toroidal boundary which is not homeomorphic to $S^1 \times D^2$. 
  Then we get for any element $\phi \in H^1(M;\IQ)$. 
\[
\deg\bigl(\overline{\rho}^{(2)}({M};\phi)\bigr)  = - x_M(\phi).
\]
\end{theorem}

Actually we get a much more general result, where we can consider not only the universal
covering but appropriate $G$-coverings $G \to \overline{M} \to M$ and get estimates for
the $L^2$-function for all times $t \in (0,\infty)$ which imply the equality of the
degree and the Thurston norm, see Theorem~\ref{the:Lower_and_upper_bounds}.

The main ingredients for the proof are the estimates for the twisted $L^2$-torsion
function and approximation techniques presented in~\cite{Lueck(2015twisting)},
the proof of the Virtual Fibering Theorem of Agol~\cite{Agol(2008),Agol(2013)}, Wise~\cite{Wise(2012hierachy)}
and Przytycki-Wise~\cite{Przytycki-Wise(2012)}
and a careful analysis of Mahler measures.
\\

\noindent \emph{Added in proof.} We just learned that Liu~\cite{Liu(2015)} has given 
a completely independent proof of Theorem~\ref{the:main_result_introduction}. 
The techniques used in both papers are at times somewhat similar. 
Liu~\cite{Liu(2015)} goes on to prove several other very interesting results that are not 
covered in this paper. In particular he proves Theorem~\ref{the:main_result_introduction} also for real classes
and shows the continuity of the $L^2$-torsion function.


\subsection*{Conventions and notations.}
We view elements in $\IZ[G]^k$ always as row vectors.  Given a group $G$ and an $m\times
n$-matrix over $\IZ[G]$ we denote by the $r_A$ the homomorphism $\IZ[G]^m\to \IZ[G]^n$
given by right multiplication by $A$. Furthermore, given a group homomorphism
$\gamma\colon G\to H$ we denote by $\gamma(A)$ the matrix over $\IZ[H]$ given by applying
$\gamma$ to all entries. Throughout the paper we assume that all 3-manifolds are compact,
connected and oriented, unless we say explicitly otherwise.

\subsection*{Acknowledgments.}
The first author gratefully acknowledges the support provided by the SFB 1085 ``Higher
Invariants'' at the University of Regensburg, funded by the Deutsche Forschungsgemeinschaft {DFG}. 
The paper is financially supported by the Leibniz-Preis of the second author granted by the {DFG}
and the ERC Advanced Grant ``KL2MG-interactions'' (no.  662400) of the second author granted by the European Research Council.
We  thank Yi Liu for helpful comments.



 \typeout{------   Section 1: Review of the $\phi$-twisted $L^2$-torsion function and Thurston norm   --------}

\section{Review of the  $\phi$-twisted $L^2$-torsion function  and the Thurston norm}
\label{sec:Review_of_the_phi-twisted_L2-torsion_function_and_the_Thurston_norm}

In this section we recall some basic definitions, notions and results
from~\cite{Dubois-Friedl-Lueck(2014Alexander),Dubois-Friedl-Lueck(2014symmetric),Lueck(2015twisting)},
in order to keep this paper self-contained. For basic information about $L^2$-Betti
numbers, Fuglede-Kadison determinants and $L^2$-torsion we refer to~\cite{Lueck(2002)}.


\subsection{Euler structures and $\operatorname{Spin}^c$-structures} 
\label{sub:Base_refinement}
Let $X$ be a finite CW complex and let $p\colon \widetilde{X}\to X$ be the universal
covering of $X$.  Following Turaev~\cite{Turaev(1990),Turaev(2001),Turaev(2002)}, we define
a \emph{fundamental family of cells} to be a choice for each open cell in $X$ of precisely
one open cell in $\widetilde{X}$ which projects to the given cell in $X$.

We write $\pi=\pi_1(X)$ and we denote by $\psi\colon \pi\to H_1(\pi;\IZ)=H_1(X;\IZ)$ the
abelianization map. Now let $\{e_i\}_{i\in I}$ and $\{\hat{e}_i\}_{i\in }$ be two
fundamental families of cells. After reordering them we can arrange that for each $i\in I$
we have $e_i=g_i\hat{e}_i$ for some $g_i\in G$. We say that two fundamental families of
cells are \emph{equivalent} if
\[ \sum_{i\in I}(-1)^{\operatorname{dim}(e_i)}\psi(g_i)=0.\] The set of equivalence
classes of fundamental families of cells on $X$ is called the set $\eul(X)$ of \emph{Euler
  structures on $X$}. Note that the set of Euler structures on $X$ admits a free and
transitive action by $H_1(X;\IZ)$.

We recall some basic facts regarding $\spinc$-structures on 3-manifolds, with empty or
toroidal boundary. We refer to~\cite[Chapter~XI]{Turaev(2002)} for a detailed discussion.
Given a 3-manifold $M$ we denote by $\spinc(M)$ the set of $\spinc$-structures on $M$.
The set $\spinc(M)$ comes with a canonical free and transitive action by $H_1(M;\IZ)$.
Given $\mathfrak{s}\in \spinc(M)$ we denote by $c_1(\mathfrak{s})\in H^2(M,\partial
M;\IZ)=H_1(M;\IZ)$ its Chern class.  The Chern class has the property that for each
$\mathfrak{s}\in \spinc(M)$ and $h\in H_1(M;\IZ)$ the following equality holds

\begin{equation} \label{equ:c1he} c_1(h\mathfrak{s})=2h+c_1(\mathfrak{s}).\end{equation}

In~\cite{Turaev(2002),Turaev(1997)} Turaev shows that given any CW-structure $X$ for $M$
there exists a canonical $H_1(M;\IZ)=H_1(X;\IZ)$-equivariant bijection $\spinc(M)\to
\eul(X)$.


\subsection{$(L^2$-acyclic) admissible pairs and the $\phi$-twisted $L^2$-torsion function}
\label{subsec:L2-acyclic_admissible_triple}
\label{subsec:The_phi-twisted_L2-torsion_function}

In~\cite{Dubois-Friedl-Lueck(2014Alexander),Dubois-Friedl-Lueck(2014symmetric)} the
authors and Dubois introduced the $\phi$-twisted $L^2$-torsion function of a
3-manifold. This definition was later generalized and analyzed
in~\cite[Section~7]{Lueck(2015twisting)} for $G$-coverings of compact connected manifolds
in all dimensions.

We start out with the following definitions.

\begin{definition}
\begin{enumerate}
\item In the following, given  any abelian group $A$ we will denote
\begin{eqnarray*}
A_f & = & A/\tors(A).
\label{A_f}
\end{eqnarray*}
\item We say that a group homomorphism $\mu \colon \pi \to G$ is \emph{\largehomo}, if the
  projection map $\pi\to H_1(\pi;\IZ)_f$ factors through $\mu$.
\item  An   \emph{admissible pair} $(M,\mu)$ consists of an irreducible
  $3$-manifold $M\ne S^1\times D^2$ with infinite fundamental group and a \largehomoblank
  group homomorphism $\mu \colon \pi_1(M) \to G$  to a residually finite countable  group $G$.
     Denote by $\overline{M} \to M$ the $G$-covering
    associated to $\mu$.  We say that $(M,\mu)$ is \emph{$L^2$-acyclic} if the $n$-th
    $L^2$-Betti number $b_n^{(2)}(\overline{M};\caln(G))$ vanishes for every $n \ge 0$.
\end{enumerate}
\end{definition}

Many of the subsequent results will hold in more general situations, e.g., it is not
always necessary to assume that $G$ is residually finite or that $\mu$ is \largehomo. 
Nonetheless, in an attempt to keep the paper readable we will note state all the results
in the maximal generality.

\begin{convention}
  If $\mu \colon\pi\to G$ is a \largehomoblank epimorphism, then we can and will identify
  $\operatorname{Hom}(\pi,\IR)$ with $\operatorname{Hom}(G,\IR)$. Furthermore, given any
  space $X$ we make the usual identifications
  $H^1(X;\IR)=\operatorname{Hom}(H_1(X;\IZ),\IR)=\operatorname{Hom}(\pi_1(X),\IR)$. In
  particular, if $(M,\mu\colon \pi_1(M)\to G)$ is an admissible pair, such that the
  cokernel of $\mu$ is finite, then any $\phi\in H^1(M;\IR)$ induces a unique homomorphism $G\to  \IR$ that, 
  by a slight abuse of notation, we also denote by $\phi$.

\end{convention}

The following lemma is an immediate consequence of \cite{Hempel(1987)},
\cite{Lott-Lueck(1995)} and the proof of the Geometrization Theorem by Perelman.

\begin{lemma} \label{lem:admissible_triple–universal_covering}
If $M\ne S^1\times D^2$ is an irreducible  $3$-manifold
with infinite fundamental group,  then $(M,\id_{\pi_1(M)})$ is an $L^2$-acyclic admissible pair.
\end{lemma}

Now consider an $L^2$-acyclic admissible pair $(M;\mu\colon \pi_1(M)\to G)$ with
$\spinc$-structure $\mathfrak{s}\in \spinc(M)$. Let $\phi\in H^1(M;\IQ)$.  We pick a
CW-structure for $M$, which by abuse of notation we denote again by $M$.  We denote by
$\widetilde{M}$ the universal cover of $M$ and we write $\pi=\pi_1(M)$.  We pick a
fundamental family of cells in $\widetilde{M}$ that corresponds to $\mathfrak{s}$.

This fundamental family of cells turns $C_*(\widetilde{M})$ into a chain complex of based
free $\IZ[\pi]$-left modules. (The basis is now unique up to permutation and multiplying each element with $\pm -1$ what 
will not affect the Hilbert space structure and hence the $\phi$-twisted $L^2$-torsion function below.) 
We view $\IZ[G]$ as a right $\IZ[\pi]$-module via the
homomorphism $\mu$. We obtain the chain complex
$\IZ[G]\otimes_{\IZ[\pi]}C_*(\widetilde{M})$ of based free $\IZ[G]$-left modules.

Now let $t\in (0,\infty)$. We denote by $\phi^*\IC_t$ the based $1$-dimensional complex
$G$-representation whose underlying complex vector space is $\IC$ and on which $g \in G$
acts by multiplication with $t^{\phi(g)}$. Twisting with $\phi^*\IC_t$ transforms a $\IC
G$-homomorphism $\IC G \to \IC G$ given by right multiplication with the element $\sum_{g
  \in G} \lambda_g \cdot g$ to the $\IC G$-homomorphism $\IC G \to \IC G$ given by right
multiplication with the element $\sum_{g \in G} \lambda_G \cdot t^{\phi(g)} \cdot g$. It
is obvious how this extends to $\IC G$-left linear maps $\IC G^m \to \IC G^n$ and then to
$\IC G\otimes_{\IZ \pi }C_*(\widetilde{M})$.  Thus twisting $\IC G\otimes_{\IZ \pi
}C_*(\widetilde{M})$ with $\phi^*\IC_t$ yields a finite free $\IC G$-chain complex
$\eta_{\phi^*\IC_t}(\IC G\otimes_{\IZ \pi}C_*(\widetilde{M}))$ with a $\IC G$-basis.

Given a $\IC G$-linear map $A\colon \IC G^m \to \IC G^n$, we obtain by applying $L^2(G)
\otimes_{\IC G} -$ a morphism $\Lambda^G(A)$ of finitely generated Hilbert $\caln(G)$-modules 
$L^2(G)^m \to L^2(G)^n$.  Thus we obtain from $\eta_{\phi^*\IC_t}(\IC G\otimes_{\IZ \pi}C_*(\widetilde{M}))$ by
applying $L^2(G) \otimes_{\IC G} -$ a finite Hilbert $\caln(G)$-chain complex denoted by
$\Lambda^G \circ \eta_{\phi^*\IC_t}(\IC G\otimes_{\IZ \pi}C_*(\widetilde{M}))$.
By our hypothesis this chain complex is $\det$-$L^2$-acyclic 
(in the sense of~\cite[Definition~3.29 on page~140]{Lueck(2002)})
for $t=1$. By~\cite[Theorem~6.7]{Lueck(2015twisting)} we know that it is then also  $\det$-$L^2$-acyclic for any $t\in (0,\infty)$. 
In particular the $\caln(G)$-chain complex has well-defined $L^2$-torsion for any $t\in (0,\infty)$.   
Define the \emph{$\phi$-twisted $L^2$-torsion function}
  \begin{equation}
  \begin{array}{rcl}
     \rho^{(2)}(M,\mathfrak{s};\mu,\phi) \colon (0,\infty)& \to& \IR \\
     t&\mapsto & \rho^{(2)}\Bigl(\Lambda^G \circ \eta_{(\phi \circ \nu)^*\IC_t} (\IC G\otimes_{\IZ \pi}C_*(\widetilde{M}))\Bigr).
   \label{rho(2)(M,[B_M];mu,phi)}
  \end{array}
 \end{equation}
for any choice of homomorphism $\nu \colon G \to H_1(\pi;\IZ)$ such that 
$\nu \circ \mu$ agrees with the projection
$\pi \to H_1(\pi;\IZ)$. The right handside is indeed independent of the
choice of $\nu$. Namely, if $G'$ is the image of $\mu$ and $\mu' \colon \pi \to G'$ is the epimorphism induced by $\mu$,
then there is precisley one homomorphism $\nu' \colon G' \to H_1(\pi;\IZ)$ such that $\nu' \circ \mu'$ agrees with the projection
$\pi \to H_1(\pi;\IZ)$ and we get 
\[
\rho^{(2)}\Bigl(\Lambda^G \circ \eta_{(\phi \circ \nu)^*\IC_t} (\IC G\otimes_{\IZ \pi}C_*(\widetilde{M}))\Bigr)
= \rho^{(2)}\Bigl(\Lambda^{G'} \circ \eta_{(\phi \circ \nu')^*\IC_t} (\IC G'\otimes_{\IZ \pi}C_*(\widetilde{M}))\Bigr)
\]
from~\cite[Theorem~7.7~(7)]{Lueck(2015twisting)}.
More details of this construction and the proof that it  is
well-defined
can be found in~\cite[Section~7]{Lueck(2015twisting)} and, with slightly different conventions, in~\cite{Dubois-Friedl-Lueck(2014Alexander)}.

If $\mu$ is the identity homomorphism, then we drop it from the notation.
Put differently, we write $ \rho^{(2)}(M,\mathfrak{s};\phi) := \rho^{(2)}(M,\mathfrak{s};\id_{\pi_1(M)},\phi) $.


\subsection{Comparing the $\phi$-twisted $L^2$-torsion function and the $L^2$-Alexander torsion}
\label{subsec:Comparing_the_phi-twisted_L2-torsion_function_and_the_L2_Alexander_torsion}

The $\phi$-twisted $L^2$-torsion function ${\rho}^{(2)}(M,\mathfrak{s};\mu,\phi) \colon
(0,\infty) \to \IR$, as considered in this paper and in~\cite{Lueck(2015twisting)},
is designed in an additive setup, as it is  the main convention when dealing with
related invariants such as topological $L^2$-torsion, analytic $L^2$-torsion, analytic
Ray-Singer torsion and so on.  When dealing with torsion invariants in dimension $3$, the
multiplicative setting is  standard, which is the reason why we defined for instance
in~\cite{Dubois-Friedl-Lueck(2014Alexander),Dubois-Friedl-Lueck(2014symmetric)} the $L^2$-Alexander torsion multiplicatively
as a function $\tau^{(2)}(M,\mathfrak{s};\phi,\mu) \colon (0,\infty) \to [0,\infty)$. 


If $(M,\mu)$ is  $L^2$-acyclic, then it follows immediately from 
the definitions and the conventions used in the various papers and from~\cite{Lueck(2015twisting)}, that these two invariants  are related by the formula
\begin{eqnarray}
\tau^{(2)}(M,\mathfrak{s};\phi,\mu)
& = & 
\exp\bigl(-{\rho}^{(2)}(M,\mathfrak{s};\mu,\phi)\bigr).
\label{relation_between_rho_and_tau}
\end{eqnarray}
Notice that~\eqref{relation_between_rho_and_tau} implies that $\tau^{(2)}(M,\mathfrak{s};\phi,\mu)$
never takes the value zero.  This is a consequence of
Theorem~\ref{the:Properties_of_the_twisted_L2-torsion_function}~(1) which was not
available when~\cite{Dubois-Friedl-Lueck(2014Alexander)} was finished.

Notice the minus sign appearing in the formula~\eqref{relation_between_rho_and_tau}. 
This has  the consequence that
the degree $\deg\bigl(\tau^{(2)}(M,\mathfrak{s};\phi,\mu) \bigr)$ defined 
in~\cite{Dubois-Friedl-Lueck(2014Alexander)}
and the degree $\deg\bigl({\rho}^{(2)}(M,\mathfrak{s};\mu,\phi)\bigr)$ 
defined in the introduction and later again in~\eqref{degree(overline(rho)(2)(M,;mu,phi))}
are related by
\begin{eqnarray}
\deg\bigl(\tau^{(2)}(M,\mathfrak{s};\phi,\mu)\bigr)
& = & 
-\deg\bigl({\rho}^{(2)}(M,\mathfrak{s};\mu,\phi)\bigr).
\label{relation_between_the_degrees_of_rho_and_tau}
\end{eqnarray}
In the following we will cite results from
\cite{Dubois-Friedl-Lueck(2014Alexander),Dubois-Friedl-Lueck(2014symmetric)} about
$\tau^{(2)}(M,\mathfrak{s};\phi,\mu)$, which via
(\ref{relation_between_the_degrees_of_rho_and_tau}) we reinterpret as results on
${\rho}^{(2)}(M,\mathfrak{s};\mu,\phi)$.


\subsection{Properties of  the $\phi$-twisted $L^2$-torsion function}

The following theorem summarizes some of the key properties of the $\phi$-twisted $L^2$-torsion function.

\begin{theorem}[Properties of the twisted $L^2$-torsion function]
\label{the:Properties_of_the_twisted_L2-torsion_function}
Let $(M,\mu)$ be an $L^2$-acyclic admissible pair, let $\phi\in H^1(M;\IR)$  and let $\mathfrak{s}\in  \spinc(M)$.

\begin{enumerate}
\item[$(1)$] \label{the:Properties_of_the_twisted_L2-torsion_function:determinant_class} 
\emph{Pinching estimate}\\
There exist constants $C$ and $D$ such that we get  for  $0 < t \le 1$
\[
C \cdot \ln(t) -D
\le 
\rho^{(2)}(M,\mathfrak{s};\mu,\phi)(t)
\le
- C \cdot \ln(t) + D,
\]
and for  $t \ge 1$
\[
- C \cdot \ln(t)  -D
\le 
\rho^{(2)}(M,\mathfrak{s};\mu,\phi)(t)
\le
C  \cdot \ln(t) + D;
\]

\item[$(2)$] \emph{Dependence on the $\spinc$-structure}\\ For any $h\in H_1(M;\IZ)$ we have 
\[ \rho^{(2)}(M,h\mathfrak{s};\mu,\phi)=\rho^{(2)}(M,\mathfrak{s};\mu,\phi)+\ln(t)\cdot \phi(h).\]
\item[$(3)$] \emph{Covering formula}\\
  Let $p\colon \widehat{M}\to M$ be a finite regular covering such that $\ker(\mu)\subset
  \widehat{\pi}:=\pi_1(\widehat{M})$. We write $\widehat{\phi}:=p^*\phi$ and we denote by
  $\widehat{\mu}$ the restriction of $\mu$ to $\widehat{\pi}$. Finally we write
  $\widehat{\mathfrak{s}}:=p^*(\mathfrak{s})$. Then for all $t$ we have
\[\rho^{(2)}(\widehat{M},\widehat{\mathfrak{s}};\widehat{\phi},\widehat{\mu})(t)=[\widehat{N}:N]\cdot \rho^{(2)}(M,\mathfrak{s},\phi,\mu)(t).\]
\item[$(4)$] \label{the:Properties_of_the_twisted_L2-torsion_function:scaling} \emph{Scaling $\phi$}\\
Let $r \in \IR$. 
Then we get for all $t \in (0,\infty)$
\[
\rho^{(2)}(M,\mathfrak{s};\mu,r\phi)(t) = \rho^{(2)}(M,\mathfrak{s};\mu,\phi)(t^r).
\]  
\item[$(5)$] \emph{Symmetry}\\
For any $t\in (0,\infty)$ we have 
\[ \rho(M,\mathfrak{s};\mu,\phi)(t^{-1})=-\phi(c_1(\mathfrak{s}))\ln(t)+\rho(M,\mathfrak{s};\mu,\phi)(t).\]
\end{enumerate}
\end{theorem}

Statement (1) is proved in~\cite[Theorem~7.4~(i)]{Lueck(2015twisting)}, it is one of the
main results of that paper. Statement (2) is proved in
\cite{Dubois-Friedl-Lueck(2014Alexander)} and \cite{Dubois-Friedl-Lueck(2014symmetric)}.
Statement (3) is proved in~\cite[Theorem~5.7~(6)]{Lueck(2015twisting)}
and~\cite[Lemma~5.3]{Dubois-Friedl-Lueck(2014Alexander)} without explicitly mentioning
$\spinc$-structures. Nonetheless, it is straightforward to see that the proofs provided in
the literature also imply the statement about $\spinc$-structures.  Statement (4) is
basically a tautology, see~\cite[Theorem~7.4~(5)]{Lueck(2015twisting)}
and~\cite[Lemma~5.2]{Dubois-Friedl-Lueck(2014Alexander)}.  Finally Statement (5) is
obtained in the proof of Theorem~1.1 of~\cite{Dubois-Friedl-Lueck(2014symmetric)}.

Define two functions $f_0 , f_1 \colon (0,\infty) \to \IR$ to be \emph{equivalent} if there is an 
$m \in \IR$ such that $f_1(t) - f_0(t) = m \cdot \ln(t)$ holds. Because of 
Theorem~\ref{the:Properties_of_the_twisted_L2-torsion_function}~(2) the equivalence
class of the function $\rho^{(2)}(M,\mathfrak{s};\mu,\phi)$ defined in~\eqref{rho(2)(M,[B_M];mu,phi)}
is independent of the choice of the $\spinc$-structure, and will be denoted by 
\begin{eqnarray}
& \overline{\rho}^{(2)}(M;\mu,\phi). & 
\label{overline(rho)(2)(M,;mu,phi)}
\end{eqnarray}

Theorem~\ref{the:Properties_of_the_twisted_L2-torsion_function}~%
 (1)
allows us to define the degree of $\overline{\rho}^{(2)}(M;\mu,\phi)$ by
\begin{eqnarray}
\deg\bigl(\overline{\rho}^{(2)}(M;\mu,\phi)\bigr) 
& = & 
\limsup_{t \to \infty} \frac{\rho(t)}{\ln(t)} - \liminf_{t \to 0}  \frac{\rho(t)}{\ln(t)}
\label{degree(overline(rho)(2)(M,;mu,phi))}
\end{eqnarray}
for any representative $\rho \colon (0,\infty) \to \IR$ of $\overline{\rho}^{(2)}(M;\mu,\phi)$.


\subsection{Approximation}
\label{subsec:Approximation}
The following is a consequence of one of the main technical results 
of~\cite{Lueck(2015twisting)}.

\begin{theorem}[Twisted Approximation inequality]
\label{the:Twisted_Approximation_inequality}
Let $\phi \colon G \to \IR$  be  a group homomorphism whose  image is finitely generated.

Consider a nested sequence of normal subgroups of $G$
\[
G  \supseteq G_0 \supseteq G_1 \supseteq G_2 \supseteq \cdots 
\]
such that $G_i$ is contained in $\ker(\phi) $ and the intersection $\bigcap_{i \ge 0} G_i$
is trivial. Suppose  that the index $[\ker(\phi) : G_i]$ is finite for all $i \ge 0$.
Put $Q_i := G/G_i$.  Let $\phi_i \colon Q_i \to \IR$ be the homomorphism
uniquely determined by $\phi_i \circ \pr_i = \phi$, where $\pr_i \colon G \to Q_i$ is the
canonical projection.

Fix an $(r,s)$-matrix $A \in M_{r,s}(\IZ G)$.
Denote by $A[i]$ the image of $A$ under the map $M_{r,s}(\IZ G) \to M_{r,s}(\IZ Q_i)$ 
induced by the projection $\pr_i$.

Then we get 
\begin{align*}
\begin{array}{rcl}
\dim_{\caln(G)}\bigl(\ker(\Lambda^G \circ \eta_{\phi^* \IC_t}(r_A))\bigr)
& = & 
\lim_{i \to \infty} \dim_{\caln(Q_i)}\bigl(\ker(\Lambda^{Q_i} \circ \eta_{\phi_i^* \IC_t}(r_{A[i]}))\bigr)
\intertext{and}
{\det}_{\caln(G)}\bigl(\Lambda^G \circ \eta_{\phi^* \IC_t}(r_A)\bigr)
& \ge  & 
\limsup_{i \to \infty} {\det}_{\caln(Q_i)}\bigl(\Lambda^{Q_i} \circ \eta_{\phi_i^*\IC_t}(r_{A[i]})\bigr).
\end{array}\end{align*}
\end{theorem}
\begin{proof}
Since the image of $\phi$ is finitely generated, we can choose a monomorphism $j \colon \IZ^d \to \IR$ 
and an epimorphism $\phi' \colon G \to \IZ^d$ with $\phi = j \circ \phi$. 
Now we apply~\cite[Theorem~6.52]{Lueck(2015twisting)} 
to $\phi'$ in the special case $V = j^*\IC_t$.
\end{proof}


\subsection{The Thurston norm}
\label{subsec:Thurston_norm}

Recall  the definition in~\cite{Thurston(1986norm)} of the \emph{Thurston norm} $x_M(\phi)$ 
of a  $3$-manifold $M$
and  an element  $\phi\in H^1(M;\IZ)=\operatorname{Hom}(\pi_1(M),\IZ)$ 
 \[
x(\phi):=\min \{ \chi_-(F)\, | \, F \subset N \mbox{ properly embedded surface dual to }\phi\},
\]
where, given a surface $F$ with connected components $F_1, F_2, \ldots , F_k$, we define
\[\chi_-(F)=\sum_{i=1}^k \max\{-\chi(F_i),0\}.\]

Thurston~\cite{Thurston(1986norm)} showed that this defines a seminorm on
$H^1(M;\mathbb{Z})$ which can be extended to a seminorm on $H^1(M;\mathbb{R} )$ which we
also denote by $x_M$ again.  In particular we get for $r \in \IR$  and $\phi \in H^1(N;\IR)$
\begin{eqnarray}
x_M(r \cdot \phi) 
& = & 
|r| \cdot x_M(\phi).
\label{scaling_Thurston_norm}
\end{eqnarray}
If $p \colon N \to M$ is a finite covering with $n$ sheets, then
Gabai~\cite[Corollary~6.13]{Gabai(1983)} showed that
\begin{eqnarray}
x_N(p^*\phi) 
& = & 
n \cdot x_M(\phi).
\label{finite_coverings_Thurston_norm}
\end{eqnarray}
If $F \to M \xrightarrow{p} S^1$ is a fiber bundle for a  $3$-manifold $M$
and compact  surface $F$
and $\phi \in H^1(M;\IZ)$ is given by $H_1(p) \colon H_1(M) \to H_1(S^1)=\IZ$, then by~\cite[Section~3]{Thurston(1986norm)} we have 
\begin{eqnarray}
x_M(\phi) & = & 
\begin{cases}
- \chi(F) & \text{if} \;\chi(F) \le 0;
\\
0 & \text{if} \;\chi(F) \ge 0.
\end{cases}
\label{x_for_fiber_bundles}
\end{eqnarray}


 \typeout{----   Section 2: Fox calculus and the $\phi$-twisted $L^2$-torsion function in terms of the second differential------}

\section{Calculating the $\phi$-twisted $L^2$-torsion function}
\label{sec:Fox_calculus_and_the_phi-twisted_L2-_torsion_function_in_terms_of_the_second_differential}

The following theorem says that given $M$ and $\psi\in H^1(M;\IQ)$ the corresponding $L^2$-torsion functions 
can be computed using one fixed square matrix over $\IZ[\pi_1(M)]$ together with a well-understood error term.

\begin{theorem}
 \label{the:calculation_of_L2-torsion_from_a_presentation} 
Let $M$ be a $3$-manifold with $b_1(M)>0$, let $\mathfrak{s}\in  \spinc(M)$. We write $\pi=\pi_1(M)$.
\begin{enumerate}

\item[$(1)$] \label{the:calculation_of_L2-torsion_from_a_presentation:non-empty} Suppose
  $\partial M$ is non-empty and toroidal. Then there exists an $s\in \pi_1(M)$ and a
  square matrix $A$ over $\IZ[\pi]$ such that the following conditions are satisfied for
  any \largehomoblank homomorphism $\mu\colon \pi\to G$ and any homomorphism $\phi\colon
  G\to \IR$:
\begin{enumerate}
\item[$(a)$]  $b_n^{(2)}(\overline{M};\caln(G)) = 0$ holds for all $n \ge 0$ if and only if
$\dim_{\caln(G)}\bigl(\ker(\Lambda^G \circ \eta_{\phi^* \IC_t}(r_A))\bigr)$ vanishes for all $t > 0$.
\item[$(b)$] If  $(a)$  is the case, then $(M,\mu)$ is  $\phi$-twisted $\det$-$L^2$-acyclic $($in the sense
of~\cite[Definition~7.1]{Lueck(2015twisting)}$)$ 
and we get 
\[ \hspace{1cm}\rho^{(2)}(M,\mathfrak{s};\mu,\phi)(t) 
=  - \ln \bigl({\det}_{\caln(G)}(\Lambda^G \circ \eta_{\phi^* \IC_t}(r_A))\bigr)+\eta(t)\]
where $\eta(t)$ is given by
\[
\eta(t)=\max\{0,|\phi(s)| \cdot \ln(t)\}.\]
\end{enumerate}
\item[$(2)$] \label{the:calculation_of_L2-torsion_from_a_presentation:empty} Suppose $M$
  is closed. Then there exist $s,s'\in \pi_1(M)$ and a square matrix $A$ over $\IZ[\pi]$
  such that the following conditions are satisfied for any \largehomoblank homomorphism
  $\mu\colon \pi\to G$ and any homomorphism $\phi\colon G\to \IR$:
\begin{enumerate}
\item[$(a)$]  $b_n^{(2)}(\overline{M};\caln(G)) = 0$ holds for all $n \ge 0$ if and only if
$\dim_{\caln(G)}\bigl(\ker(\Lambda^G \circ \eta_{\phi^* \IC_t}(r_A))\bigr)$ vanishes for all $t > 0$.
\item[$(b)$] If $(a)$  is the case, then $(M,\mu)$ is  $\phi$-twisted $\det$-$L^2$-acyclic and we get 
\[ \hspace{1cm}\rho^{(2)}(M,\mathfrak{s};\mu,\phi)(t) 
=  - \ln \bigl({\det}_{\caln(G)}(\Lambda^G \circ \eta_{\phi^* \IC_t}(r_A))\bigr)+\eta(t)\]
where $\eta(t)$ is given by
\[\hspace{1cm}
\eta(t)=\max\{0,|\phi(s)| \cdot \ln(t)\}+\max\{0,|\phi(s')| \cdot \ln(t)\}.\]
\end{enumerate}
\end{enumerate}
\end{theorem}

\begin{proof}
  We only treat the case, where $\partial M$ is empty, and leave it to the reader to
  figure out the details for the case of a non-empty boundary using the proof
  of~\cite[Theorem~2.4]{Lueck(1994a)}.   From~\cite[Proof of Theorem~5.1]{McMullen(2002)} we
  obtain the following:
\begin{enumerate}
  \item a compact $3$-dimensional $CW$-complex $X$ together with a homeomorphism
  $f \colon X \to M$ (in the following we identify $\pi = \pi_1(M) = \pi_1(X)$ using $\pi_1(f)$),
  \item  two sets of generators $\{s_1,\dots,s_a\}$ and $\{s_1', \dots ,s_a'\}$ of $\pi$,
  \item an $a\times a$-matrix $F$ over $\IZ[\pi]$,
  \end{enumerate}
such the cellular
  $\IZ \pi$-chain complex $C_*(\widetilde{X})$ of the universal cover $\widetilde{X}$ looks for an appropriate
 fundamental family of cells
  like
\[
\IZ \pi
\xrightarrow{\prod\limits_{i=1}^a r_{s_i' - 1}} 
\bigoplus_{i=1}^a \IZ \pi 
\xrightarrow{r_{F}} \bigoplus_{i=1}^a \IZ \pi 
\xrightarrow{\bigoplus\limits_{i=1}^a r_{s_i - 1}} \IZ \pi.
\]
The based $\IZ G$-chain complex $\IZ G\otimes_{\IZ \pi]}C_*(\widetilde{X})$ looks like 
\[
\IZ G 
\xrightarrow{\prod\limits_{i=1}^a r_{\mu(s_i') - 1}} 
\bigoplus_{i=1}^a \IZ G 
\xrightarrow{r_{\mu(F)}} \bigoplus_{i=1}^a \IZ G 
\xrightarrow{\bigoplus\limits_{i=1}^a r_{\mu(s_i) - 1}} \IZ G.
\]
Then the Hilbert $\caln(G)$-chain complex $\Lambda^G
\circ \eta_{\phi^*\IC_t}(C_*(\overline{X}))$ looks like
\begin{multline*}
L^2(G)  
\xrightarrow{{\prod\limits_{i=1}^a \Lambda^G\big(r_{t^{\phi(s_1')} \cdot \mu(s_i')-1}\big)}}
\bigoplus_{i=1}^a L^2(G)
\xrightarrow{\Lambda^G \circ \eta_{\phi^*\IC_t}(r_{\mu(F)})} \bigoplus_{i=1}^a L^2(G)  
\\
\xrightarrow{\bigoplus\limits_{i=1}^a \Lambda^G\big(r_{t^{\phi(s_i)} \cdot \mu(s_i)- 1}\big)} L^2(G).
\end{multline*}
Since $b_1(M)>0$ is non-trivial there exist $i,j\in\{1,\dots,a\}$ such that $s_i$ and
$s_j$ represent non-zero elements in $H_1(M;\IZ)_f$. For later we record, that given any
\largehomoblank homomorphism $\mu\colon \pi\to G$ the images $\mu(s)$ and $\mu(s')$ have
infinite order.  We denote by $A$ the matrix that is obtained from $F$ by removing the
$i$-th column and the $j$-th row.

For $g \in G$ and $t \in (0,\infty)$  let $D(g,t)_*$ be the $1$-dimensional
Hilbert $\caln(G)$-chain complex which has as first differential
$\Lambda^G(r_{t^{\phi(g)} \cdot g -1}) \colon L^2(G) \to L^2(G)$. Provided
that $|g| = \infty$ holds, $D(g,t)_*$ is $L^2$-$\det$-acyclic and a direct
computation using~\cite[Theorem~3.14~(6) on page~129 and~(3.23) on page~136]{Lueck(2002)}
shows
\begin{equation}
\rho^{(2)}(D(g,t)_*)
 = 
\ln \bigl({\det}_{\caln(G)}\bigl(\Lambda^G(r_{t^{-\phi(g)} \cdot  g -1}\bigr)\bigr)\bigr)
\label{the:calculation_of_L2-torsion_from_a_presentation:(1)} 
=  \max\{|\phi(g)| \cdot \ln(t),0\}.
\end{equation}
Now let $\mathfrak{s}\in \spinc(M)$ be the $\spinc$-structure that corresponds to the
above fundamental family of cells. It follows from
\cite[Lemma~3.2]{Dubois-Friedl-Lueck(2014Alexander)} that the above group elements $s,s'$
and the matrix $A$ have all the desired properties.

If $\mathfrak{t}\in \spinc(M)$ is a different $\spinc$--structure, then we can write
$\mathfrak{t}=h\mathfrak{s}$ for some $h\in H_1(M;\IZ)$. We pick a representative $g\in
\pi$ of $h$ and we multiply one column of $A$ by $h$ to obtain the matrix with the desired
properties.
\end{proof}


 \typeout{-------------------------   Section 3: Lower bounds -----------------------------------}

\section{Lower bounds}
\label{sec:Lower_bounds}

The elementary proof of the next lemma can be found in~\cite[Lemma~6.9]{Lueck(2015twisting)}.

\begin{lemma} \label{lem:det_estimate_in_terms_of_norm}
Let $f \colon L^2(G)^m \to L^2(G)^n$ be a bounded $G$-equivariant operator.  Then
\[
{\det}_{\caln(G)}(f) \le \|f\|^{\dim_{\caln(G)}(\overline{\im(f)})}.
\]
\end{lemma}

The next result is an improvement of~\cite[Proposition~9.5]{Dubois-Friedl-Lueck(2014Alexander)}.

\begin{lemma}\label{lem:degree_estimate}
Consider bounded $G$-equivariant operators 
$f_0,f_1 \colon L^2(G)^m \to L^2(G)^m$.
For $t > 0$ we define
\[
f[t] := f_0 + t \cdot f_1.
\]
Suppose that for every $t > 0$ the operator $f[t] \colon L^2(G)^m \to L^2(G)^m$ is $L^2$-$\det$-acyclic.
Put
\[
\rho \colon (0,\infty) \to (0,\infty), \quad t \mapsto \ln({\det}_{\caln(G)}(f[t])).
\]
Then we get
\[
\begin{array}{rclclcl}
\rho(t)
& \le & 
m \cdot \max\bigl\{0,\ln (\|f_0\| +\|f_1\|)\bigr\}
& & \text{for}\;  t  \le 1;
\\
\rho(t)
& \le  &  \dim_{\caln(G)}(\overline{\im(f_1)}) \cdot  \ln(t) + m \cdot \max\bigl\{0, \ln (2 \cdot \|f_0\| +\|f_1\|)\bigr\}
& & \text{for}\;  t \ge  1.
\end{array}
\]
In particular we get 
\begin{eqnarray*}
\limsup_{t \to \infty} \frac{\rho(t)}{\ln(t)} 
& \le & 
\dim_{\caln(G)}(\overline{\im(f_1)});
\\
\liminf_{t \to 0} \frac{\rho(t)}{\ln(t)} 
& \ge & 
0;
\\
\limsup_{t \to \infty} \frac{\rho(t)}{\ln(t)}  - \liminf_{t \to 0} \frac{\rho(t)}{\ln(t)} 
& \le & 
\dim_{\caln(G)}(\overline{\im(f_1)}).
\end{eqnarray*}
\end{lemma}

\begin{proof}
It suffices to prove the two inequalities for $\rho(t)$, then the other claims follow.

We begin with the case $t \le 1$. We get from Lemma~\ref{lem:det_estimate_in_terms_of_norm}
\begin{eqnarray*}
{\det}_{\caln(G)}(f[t])
& \le  &
\|f[t]\|^{\dim_{\caln(G)}(\overline{\im(f[t])})}.
\end{eqnarray*}
If $\|f[t]\| \le 1$, this implies ${\det}_{\caln(G)}(f[t]) \le 1$ and the claim follows.
Hence it remains to treat the case $\|f[t]\| >1$.
 Then we get because of
$\dim_{\caln(G)}(\overline{\im(f)}) \le m$ that
\begin{eqnarray*}
{\det}_{\caln(G)}(f[t])
& \le &
\|f[t]\|^m
\\
& = &
\|f_0 + t \cdot f_1\|^m
\\
& \le &
\left(\|f_0\| + t \cdot \|f_1\|\right)^m
\\
& \stackrel{t \le 1}{\le} &
\left(\|f_0\| +\|f_1\|\right)^m.
\end{eqnarray*}
Next we consider the case $t\ge 1$. We have the orthogonal decomposition
\[
L^2(G)^m =  \overline{\im(f_1)} \oplus \overline{\im(f_1)}^{\perp}.
\]
With respect to this decomposition we get for any bounded $G$-equivariant operator 
$g \colon L^2(G)^m \to L^2(G)^m$ the  decomposition 
\[
g = \begin{pmatrix}g^{(1,1)} & g^{(1,2)} \\ g^{(2,1)} & g^{(2,2)} \end{pmatrix}.
\] 
We estimate  for $t \ge 1$ using~\cite[Theorem~3.14~(1) and~(2) on page~128]{Lueck(2002)}
\[
\begin{array}{rcl}
\frac{{\det}_{\caln(G)}(f[t])}{t^{\dim_{\caln(G)}(\overline{\im(f_1)})}}
& = & 
{\det}_{\caln(G)} \left(\begin{pmatrix} t^{-1} \cdot \id & 0 \\ 0 & \id \end{pmatrix}\right) \cdot {\det}_{\caln(G)}(f[t])
\\
& = & 
{\det}_{\caln(G)} \left(\begin{pmatrix} t^{-1} \cdot \id & 0 \\ 0 & \id \end{pmatrix} \circ f[t]\right)
\\
&  \stackrel{Lemma~\ref{lem:det_estimate_in_terms_of_norm}}{\le} &
\left\|\begin{pmatrix} t^{-1} \cdot \id & 0 \\ 0 & \id \end{pmatrix} \circ f[t] \right\|^{m}.
\end{array}
\]
If $\left\|\begin{pmatrix} t^{-1} \cdot \id & 0 \\ 0 & \id \end{pmatrix} \circ f[t] \right\| \le 1$
the claim is obviously true. Hence it remains to treat the case
$\left\|\begin{pmatrix} t^{-1} \cdot \id & 0 \\ 0 & \id \end{pmatrix} \circ f[t] \right\| \ge 1$.
Then we get 
\begin{eqnarray*}
\lefteqn{\frac{{\det}_{\caln(G)}(f[t])}{t^{\dim_{\caln(G)}(\overline{\im(f_1)})}}}
\\
& \le  & 
\left\|\begin{pmatrix} t^{-1} \cdot \id  & 0 \\ 0 & \id \end{pmatrix} (f_0 +t \cdot f_1)\right\|^m
\\
& = & 
\left\|\begin{pmatrix} t^{-1}  f_0^{(1,1)} &  t^{-1}  f_0^{(1,2)} \\ f_0^{(2,1)}  & f_0^{(2,2)}  \end{pmatrix} +
\begin{pmatrix} f_1 ^{(1,1)} & f_1^{(1,2)} \\ 0 & 0  \end{pmatrix}
\right\|^m
\\
& \le  & 
\left(t^{-1} \cdot  \left\|\begin{pmatrix} f_0^{(1,1)} & f_0^{(1,2)} \\ 0 & 0 \end{pmatrix} \right\|
+ \left\|\begin{pmatrix} 0 & 0 \\  f_0^{(2,1)}  & f_0^{(2,2)}  \end{pmatrix} \right\|
+ \left\|\begin{pmatrix} f_1 ^{(1,1)} & f_1^{(1,2)} \\ 0 & 0  \end{pmatrix} \right\|\right)^m
\\
& \stackrel{t^{-1} \le 1}{\le}  & 
\left(\left\|\begin{pmatrix} f_0^{(1,1)} & f_0^{(1,2)} \\ 0 & 0   \end{pmatrix}\right\| 
+ \left\|\begin{pmatrix} 0 & 0 \\  f_0^{(2,1)}  & f_0^{(2,2)}  \end{pmatrix} \right\|
+ \left\|\begin{pmatrix} f_1 ^{(1,1)} & f_1^{(1,2)} \\ 0 & 0  \end{pmatrix} \right\|\right)^m
\\
& \le  & 
\left(\left\|\begin{pmatrix} f_0^{(1,1)} & f_0^{(1,2)} \\    f_0^{(2,1)}  & f_0^{(2,2)}   \end{pmatrix}\right\| 
+ \left\|\begin{pmatrix} f_0^{(1,1)} & f_0^{(1,2)}  \\  f_0^{(2,1)}  & f_0^{(2,2)}  \end{pmatrix} \right\|
+ \left\|\begin{pmatrix} f_1 ^{(1,1)} & f_1^{(1,2)} \\ 0 & 0  \end{pmatrix} \right\|\right)^m
\\
& = & 
\left(2 \cdot \|f_0\| + \|f_1\|\right)^m.
\end{eqnarray*}
This finishes the proof of Lemma~\ref{lem:degree_estimate}.
\end{proof}

For an element $x = \sum_{g \in G} r_g \cdot g$ in $\IC G$ define $|x|_1 := \sum_{g \in G}
|r_g|$. Given a matrix $A \in M_{r,s}(\IC G)$ define
\begin{eqnarray}
\|A\|_1 & = & r\cdot s  \cdot \max\big\{ |a_{j,k}|_1 \,\big|\, 1 \le j \le r, 1 \le k \le s\big\}.
\label{||A||_1}
\end{eqnarray}

The next theorem can be viewed as saying, that in the acyclic case the degree of the
$\phi$-twisted $L^2$-torsions gives lower bounds on the Thurston norm. This result is thus
an analogue of the classical fact, that the degree of the Alexander polynomial gives a
lower bound on the knot genus. We refer to~\cite{Dubois-Friedl-Lueck(2015flavors)} for a
detailed discussion of various twisted generalizations of the Alexander polynomial of a
knot and their relations to the Thurston norm.

\begin{theorem}[Lower bound]
\label{the:lower_bound}
Let $M$ be an irreducible $3$-manifold with infinite fundamental group $\pi$.  Let
$\mathfrak{s}\in \spinc(M)$.  Then for any $\phi\in H^1(M;\IQ)$ there exists a constant
$D\ge 0$ such that for any \largehomoblank homomorphism $\mu\colon \pi_1(M)\to G$, for
which $(M,\mu)$ is $L^2$-acyclic, we have
\begin{eqnarray*}
\tmfrac{1}{2}\big(\phi(c_1(\mathfrak{s}))+  x_M(\phi)\big)  \cdot \ln(t)  - D  & \le & \rho^{(2)}(M,\mathfrak{s};\mu,\phi)(t)  
\quad  \text{for}\; t  \le 1;
\\
\tmfrac{1}{2}\big(\phi(c_1(\mathfrak{s}))- x_M(\phi)\big)  \cdot \ln(t)  - D
& \le &
\rho^{(2)}(M,\mathfrak{s};\mu,\phi)(t)    
\quad  \text{for}\; t \ge 1.
\end{eqnarray*}
\end{theorem}

In~\cite[Theorem~1.5]{Dubois-Friedl-Lueck(2014Alexander)} we proved the analogous statement 
under the extra assumption that $\mu\colon \pi_1(M)\to G$ is a homomorphism to a virtually abelian group. 


In the proof of Theorem~\ref{the:lower_bound} we will make use of the following elementary lemma.

\begin{lemma}\label{lem:primitiveisenough}
Let $M$ be an irreducible $3$-manifold with
infinite fundamental group  and let $\mathfrak{s}\in  \spinc(M)$. 
If the conclusion of Theorem~\ref{the:lower_bound} holds for all primitive $\phi\in H^1(M;\IZ)$,  then it holds for all $\phi\in H^1(M;\IQ)$. 
\end{lemma}

\begin{proof}
  If $\phi$ is trivial, then clearly there is nothing to prove.  So let $\phi\in
  H^1(M;\IQ)$ be non-zero. We pick an $r\in \IQ_{> 0}$ such that $r\phi\in H^1(M;\IZ)$ is
  primitive. We denote by $D$ the constant of Theorem~\ref{the:lower_bound} corresponding
  to the primitive class $r\phi$.

From
Theorem~\ref{the:Properties_of_the_twisted_L2-torsion_function}~(4) and
from~\eqref{scaling_Thurston_norm} we get for any \largehomoblank homomorphism $\mu\colon \pi_1(M)\to G$,
for which  $(M,\mu)$ is $L^2$-acyclic, that 
\begin{align*}
\begin{array}{rcl}
\hspace{3cm}\rho^{(2)}(M,\mathfrak{s};\mu,\phi)(t) & = & \rho^{(2)}(M,\mathfrak{s};\mu,r\phi)\big(t^{\frac{1}{r}}\big);
\\
x_M(r\phi) & = & r \cdot x_M(\phi).
\intertext{Combining these equalities with the elementary equalities}
\ln\big(t^{\frac{1}{r}}\big)&=&\frac{1}{r}\ln(t);\\
(r\phi)(c_1(\mathfrak{s}))&=&r\cdot \phi(c_1(\mathfrak{s})),
\end{array}
\end{align*}
it is straightforward to see that the desired inequalities also hold for $\mu$ and $\phi$.
\end{proof}

\begin{proof}[Proof of Theorem~\ref{the:lower_bound}]
By Lemma~\ref{lem:primitiveisenough} it suffices to prove the statement for every primitive  $\phi\in H^1(M;\IZ)$. 
We start out with the following claim.

\begin{claim}
  Given a primitive $\phi\in H^1(M;\IZ)$ there exists an $\mathfrak{s}\in \spinc(M)$ such
  that for any \largehomoblank homomorphism $\mu\colon \pi_1(M)\to G$, for which $(M,\mu)$
  is $L^2$-acyclic, the following inequalities hold
\begin{eqnarray*}
-D  & \le & \rho^{(2)}(M,\mathfrak{s};\mu,\phi)(t)  
\quad  \text{for}\; t  \le 1;
\\
- x_M(\phi)  \cdot \ln(t)  - D 
& \le &
\rho^{(2)}(M,\mathfrak{s};\mu,\phi)(t)    
\quad  \text{for}\; t \ge 1.
\end{eqnarray*}
\end{claim}

In the following we abbreviate
\[
\rho(\mu,\phi) = \rho^{(2)}(M,\mathfrak{s};\mu,\phi).
\]
We conclude by inspecting the proof of~\cite[Proposition~9.1 in
Section~9.1]{Dubois-Friedl-Lueck(2014Alexander)} which is based
on~\cite[Section~4]{Friedl(2014twisted)}, that there exists an $\mathfrak{s}\in
\spinc(M)$, integers $k,l,m$ with $k,l \ge 0$ and $x_M(\phi) = k-l$, an element $\gamma
\in \pi$ with $\phi(\gamma) = 1$, a matrix $A \in M_{k+m,k+m}(\IZ K)$ for $K =
\ker(\phi)$, such that for any \largehomoblank homomorphism $\mu\colon \pi_1(M)\to G$,
for which  $(M,\mu)$ is $L^2$-acyclic, the following equality holds
\[
\rho(\mu,\phi)(t) = - \ln\left( \max\{1,t\}^{-l} \cdot {\det}_{\caln(G)}
\left(\Lambda^G(r_{\mu(A)}) + t \cdot \mu(\gamma) \cdot \id_{L^2(G)^k} \oplus \;0_{L^2(G)^m} \right)\right).
\]
 This implies
\[
\rho(\mu,\phi)(t) = \begin{cases}
- \ln\left({\det}_{\caln(G)}\left(\Lambda^G(r_{\mu(A)}) + t \cdot \mu(\gamma) \cdot
\id_{L^2(G)^k} \oplus \;0_{L^2(G)^m}\right)\right)
& \text{for} \; t \le 1;
\\
l \cdot \ln(t) -  \ln\left( {\det}_{\caln(G)}\left(\Lambda^G(r_{\mu(A)}) 
+ t \cdot \mu(\gamma) \cdot \id_{L^2(G)^k} \oplus \;0_{L^2(G)^m} \right)\right)
& \text{for} \; t \ge  1.
\end{cases}
\]
Define
\begin{eqnarray*}
D& = & 
(k+m) \cdot \ln\bigl(2 \cdot (\|A\|_1 + 1\bigr)\bigr).
\end{eqnarray*}
Note that $D$  depends on $\phi$ but not on $\mu$.
We conclude from~\cite[ Lemma~6.3]{Lueck(2015twisting)}  and the monotonicity of $\ln$ that 
\begin{eqnarray*}
D 
& \ge  & 
(k+m) \cdot \ln\Bigl(2\cdot \|\Lambda^G(r_{\mu(A)})\| + \|\id_{L^2(G)^k} \oplus \;0_{L^2(G)^m}\|\Bigr)
\\
 & \ge & 
(k+m) \cdot \ln\Bigl( \|\Lambda^G(r_{\mu(A)})\| + \|\id_{L^2(G)^k} \oplus \;0_{L^2(G)^m}\|\Bigr).
\end{eqnarray*}
Therefore we conclude from Lemma~\ref{lem:degree_estimate} applied 
in the case  $f_0 = \Lambda^G(r_{\mu(A)})$ and $f_1 = \mu(\gamma) \cdot \id_{L^2(G)^k} \oplus \;0_{L^2(G)^m}$ that 
\[
\ln\Big( {\det}_{\caln(G)}\left(\Lambda^G(r_{\mu(A)}) + t \cdot \mu(\gamma) \cdot \id_{L^2(G)^k} \oplus \;0_{L^2(G)^m} \right)\Big) 
\le 
\begin{cases} D &  t \le 1;
\\
k \cdot \ln(t) + D & t \ge t.
\end{cases}
\]
This implies
\begin{eqnarray*}
- D  & \le & \rho(\mu,\phi)(t)  
\quad  \text{for}\; t \le 1;
\\
- (k-l) \cdot \ln(t)  - D 
& \le &
\rho(\mu,\phi)(t)   
\quad  \text{for}\; t \ge 1.
\end{eqnarray*}
Since $x_M(\phi)= k-l$, this implies the claim.

We now turn to the proof of the desired inequalities in the theorem. Using
Theorem~\ref{the:Properties_of_the_twisted_L2-torsion_function} (2) and 
equality~\eqref{equ:c1he} one can easily see that if the desired inequalities hold for one
$\spinc$-structure of $M$, then they also hold for all other $\spinc$-structures of
$M$. Now let $\mathfrak{s}\in \spinc(M)$ be the Euler structure from the claim. Then
\begin{eqnarray*}
-D  & \le & \rho^{(2)}(M,\mathfrak{s};\mu,\phi)(t)  
\quad  \text{for}\; t  \le 1;
\\
- x_M(\phi)  \cdot \ln(t)  - D 
& \le &
\rho^{(2)}(M,\mathfrak{s};\mu,\phi)(t)    
\quad  \text{for}\; t \ge 1.
\end{eqnarray*}
By Theorem~\ref{the:Properties_of_the_twisted_L2-torsion_function} (5) we also know that 
\[ \rho(M,\mathfrak{s};\mu,\phi)(t)=\phi(c_1(\mathfrak{s}))\ln(t)+\rho(M,\mathfrak{s};\mu,\phi)(t^{-1})\]
for all $t\in (0,\infty)$. 
Combining this equality with the above inequalities we obtain that 
\begin{eqnarray*}
(\phi(c_1(\mathfrak{s})) + x_M(\phi))  \cdot \ln(t)  - D  & \le & \rho^{(2)}(M,\mathfrak{s};\mu,\phi)(t)  
\quad  \text{for}\; t  \le 1;
\\
\phi(c_1(\mathfrak{s}))  \cdot \ln(t)  - D
& \le &
\rho^{(2)}(M,\mathfrak{s};\mu,\phi)(t)    
\quad  \text{for}\; t \ge 1.
\end{eqnarray*}
Adding the two inequalities for $t\leq 1$ and dividing by two, and doing the same for the
inequalities for $t\geq 1$ gives us the desired inequalities
\begin{eqnarray*}
\tmfrac{1}{2}\big(\phi(c_1(\mathfrak{s}))+  x_M(\phi)\big)  \cdot \ln(t)  - D  & \le & \rho^{(2)}(M,\mathfrak{s};\mu,\phi)(t)  
\quad  \text{for}\; t  \le 1;
\\
\tmfrac{1}{2}\big(\phi(c_1(\mathfrak{s}))- x_M(\phi)\big)  \cdot \ln(t)  - D
& \le &
\rho^{(2)}(M,\mathfrak{s};\mu,\phi)(t)    
\quad  \text{for}\; t \ge 1.
\end{eqnarray*}
\end{proof}


 \typeout{-------------------------   Section 4: Upper bounds -----------------------------------}

\section{Upper bounds}
\label{sec:Upper_bounds}

Before we can provide upper bounds on the Thurston norm we 
will need to prove one  preliminary result. This lemma will ensure that some
information which is only available at $0$ and $\infty$ leads to uniform estimates for all
$t > 0$. This will be a key ingredient when we want to apply approximation techniques.

\begin{lemma} \label{lem:uniform_estimate} Let $\phi \colon G \to \IZ$ be a non-trivial
  group homomorphism with finite kernel.  Let $A \in M_{m,m}(\IZ G)$ be a matrix such that
  $\Lambda^G(r_A) \colon L^2(G)^m \to L^2(G)^m$ is a weak isomorphism. Then
  $\Lambda^G \circ \eta_{\phi^*\IC_t}(r_A) \colon L^2(G)^m \to L^2(G)^m$ is
  $L^2$-$\det$-acyclic for any $t>0$.  Put
\[
\rho\colon(0,\infty) \to \IR, \quad  t \mapsto \ln\bigl({\det}_{\caln(G)}(\Lambda^G \circ \eta_{\phi^*\IC_t}(r_A))\bigr).
\]
Suppose that there are real numbers $C$ and $D$ and integers $k$ and $l$ such that
\begin{eqnarray*}
\lim_{t \to 0} \rho(t) - k \cdot \ln(t) = C;
\\
\lim_{t \to \infty} \rho(t) - l \cdot \ln(t) = D.
\end{eqnarray*}
Then we get for all $t > 0$
\begin{eqnarray*}
k \cdot \ln(t) + C
& \le & 
\rho(t);
\\
l \cdot \ln(t) +D
& \le & 
\rho(t).
\end{eqnarray*}
\end{lemma}
\begin{proof} 
Choose an integer  $n \ge 1 $ and an epimorphism $\phi' \colon G \to \IZ$ 
such that $\phi = n  \cdot \id_{\IZ} \circ \phi'$.  
Then we get for the two functions $\rho$ and $\rho'$
associated to $\phi$ and $\phi'$ from
Theorem~\ref{the:Properties_of_the_twisted_L2-torsion_function}~(4)
\[
\rho'(t)  = \rho(t^n).
\]
Hence we can assume without loss of generality that $\rho$ is surjective, otherwise replace
$\phi$ by $\phi'$.

Choose a group homomorphism $s \colon \IZ \to G$ with $\phi \circ s = \id$.
Choose a map of sets $\sigma \colon \im(s)\backslash G \to G$ whose composition with the projection
$\pr \colon G \to \im(s)\backslash G$ is the identity and whose composition with
$\phi \colon G \to \IZ$ is the constant map with value $0 \in \IZ$.
Let $B \in M_{m\cdot |\ker(\phi)|,m\cdot |\ker(\phi)|}(\IZ[\IZ])$ be the matrix describing the restriction of 
$r_A \colon \IZ G^m \to \IZ G^m$ with $s$,
see~\cite[(6.40)]{Lueck(2015twisting)}.
Then a direct computation shows for all $t \in (0,\infty)$
\[
s^*(\Lambda^G \circ \eta_{\phi^*\IC_t}(r_A)) = \Lambda^{\IZ} \circ \eta_{(\phi \circ s)^*\IC_t}(r_B) \colon 
L^2(\IZ)^{m \cdot |\ker(\phi)|}  \to  L^2(\IZ)^{m\cdot |\ker(\phi)|} 
\]
where $s^*$ denotes restriction with $s$.  We get from~\cite[Theorem~3.14~(5) on page~128]{Lueck(2002)}
\[
\ln\bigl({\det}_{\caln(G)}(\Lambda^G \circ \eta_{\phi^*\IC_t}(r_A))\bigr)
= 
\frac{\ln\bigl({\det}_{\caln(\IZ)}(s^*(\Lambda^G \circ \eta_{\phi^*\IC_t}(r_A)))\bigr)}{|\ker(\phi)|}.
\]
Hence we can assume without loss of generality $G = \IZ$ and $\phi = \id_{\IZ}$, 
otherwise replace $\phi \colon G \to \IZ$ by $\phi \circ s = \id \colon  \IZ \to \IZ$ and $A$ by $B$.

One easily checks
\[
r_{\det_{\IC [\IZ]}(\eta_{\IC_t}(r_A))} = \eta_{\IC_t} \bigl(r_{\det_{\IC [\IZ]}(A)}\bigr) \colon L^2(\IZ) \to L^2(\IZ).
\]
Because of~\cite[Lemma~6.25]{Lueck(2015twisting)}
we can assume without loss of generality 
$m = 1$, otherwise replace $A$ by the $(1,1)$-matrix given by $\det_{\IC [\IZ]}(A)$.

Let $p(z) \in \IC[\IZ] = \IC[z,z^{-1}]$ be the only entry in the $(1,1)$-matrix $A$. Since
$\Lambda^{\IZ}(r_A)$ is a weak isomorphism by assumption, $p$ is non-trivial.  We can
write
  \[
  p(z) = \sum_{n = n_0}^{n_1} c_n \cdot z^n
  \]
  for integers $n_0$ and $n_1$ with 
  $n_0   \le n_1$, complex numbers $c_{n_0}, c_{n_0 +1}, \ldots, c_{n_1}$ with $c_{n_0} \not= 0$
  and $c_{n_1} \not= 0$. We can also write 
  \[p(z) = c_{n_1} \cdot z^r \cdot \prod_{i=1}^s   (z - a_i) \] 
  for an integer $s \ge 0$, non-zero complex numbers $a_1, \ldots, a_r$
  and an integer $r$.  We get from~\cite[(3.23) on  page~136]{Lueck(2002)} 
   \[
   {\det}_{\caln(\IZ)}\bigl(\Lambda^{\IZ} (r_p)\bigr)  
    =  |c_{n_1}| \cdot \prod_{\substack{i=1, \ldots, s\\|a_i| \ge 1}} |a_i|.
   \]
   For $t \in (0,\infty)$ we get 
   \[p(tz) = c_{n_1} \cdot (tz)^r \cdot \prod_{i=1}^s   (tz - a_i) =
   t^{r+s} \cdot c_{n_1} \cdot z^r \cdot \prod_{i=1}^s   \Big(z - \frac{a_i}{t}\Big) ,
   \] 
   and hence
   \[
   {\det}_{\caln(\IZ)}\bigl(\Lambda^{\IZ} (r_{p(tz)})\bigr)  
     = t^{r+s} \cdot  |c_{n_1}| \cdot \prod_{\substack{i=1, \ldots, s\\|a_i/t| \ge 1}} \left|\frac{a_i}{t}\right|
    =  t^{r+s} \cdot  |c_{n_1}| \cdot \prod_{\substack{i=1, \ldots, s\\|a_i| \ge t}} \frac{|a_i|}{t}.
    \]
   This implies for $t \in (0,\infty)$
   \begin{eqnarray}
   \rho(t) 
   & = & 
   (r+s)  \cdot \ln(t) + \ln(|c_{n_1}|) +  \sum_{\substack{i=1, \ldots, s\\|a_i| \ge t}} (\ln(|a_i|)- \ln(t)).
   \label{lem:uniform_estimate:(1)}
 \end{eqnarray}
 Define positive real numbers 
   \begin{eqnarray*}
   T_0  & = &  \min\{|a_i| \mid i =1,2, \ldots , s\};
   \\
   T_{\infty} & = & \max\{|a_i| \mid i = 1,2, \ldots , s\}.
 \end{eqnarray*}
 Then we get
   \[
   \rho(t) = 
   \begin{cases}
   r \cdot \ln(t) + \ln(|c_{n_1}|) +  \sum_{i=1}^s \ln(|a_i|) & \text{for}\; t \le T_0;
   \\
   (r+s)  \cdot \ln(t) + \ln(|c_{n_1}|) & \text{for}\; t \ge T_{\infty}.
 \end{cases}
 \]   
 Since by assumption there are real numbers $C$ and $D$ and integers $k$ and $l$ such that
\begin{eqnarray*}
\lim_{t \to 0} \rho(t) - k \cdot \ln(t) = C;
\\
\lim_{t \to \infty} \rho(t) - l \cdot \ln(t) = D,
\end{eqnarray*}
we must have $r = k $, $r + s = l$, $C = \ln(|c_{n_1}|) + \sum_{i=1}^s \ln(|a_i|)$ and 
 $D = \ln(|c_{n_1}|)$. Equation~\eqref{lem:uniform_estimate:(1)} becomes
  \begin{eqnarray*}
   \rho(t) 
   & = & 
   l \cdot \ln(t) + D +  \sum_{\substack{i=1, \ldots, s\\|a_i| \ge t}} (\ln(|a_i|)- \ln(t)).
 \end{eqnarray*}
 Since $(\ln(|a_i|)- \ln(t)) \ge 0$ for $|a_i| \ge t$, we get  $l \cdot \ln(t) + \ln(D) \le \rho(t)$ for all $t > 0$.
 We estimate for $t > 0$
  \begin{eqnarray*}
\lefteqn{k \cdot \ln(t) + C}
& &
\\
& = & 
k \cdot \ln(t) + D + \sum_{i=1}^s \ln(|a_i|) 
\\
& = & 
k \cdot \ln(t) + D + \sum_{\substack{i=1, \ldots, s\\|a_i| \ge t}} \ln(|a_i|) + \sum_{\substack{i=1, \ldots, s\\|a_i| < t}} \ln(|a_i|) 
\\
& = & 
r \cdot \ln(t) + D +  s \cdot \ln(t) + \sum_{\substack{i=1, \ldots, s\\|a_i| \ge t}} (\ln(|a_i|) - \ln(t)) 
+ \sum_{\substack{i=1, \ldots, s\\|a_i| < t}} (\ln(|a_i|) - \ln(t))
\\
& = & 
l \cdot \ln(t) + D  + \sum_{\substack{i=1, \ldots, s\\|a_i| \ge t}} (\ln(|a_i|) - \ln(t)) 
+ \sum_{\substack{i=1, \ldots, s\\|a_i| < t}} (\ln(|a_i|) - \ln(t))
\\
& \le  & 
l \cdot \ln(t) + D + \sum_{\substack{i=1, \ldots, s\\|a_i| \ge t}} (\ln(|a_i|) - \ln(t)) 
\\
& = & 
\rho(t).
\end{eqnarray*}
This finishes the proof of Lemma~\ref{lem:uniform_estimate}.
\end{proof}

\begin{definition}[Fibered classes]
  Let $M$ be a $3$-manifold and consider an element $\phi\in H^1(M;\IQ)=\operatorname{Hom}(\pi_1(M),\IQ)$. We say that
  $\phi$ is \emph{fibered} if there exists a locally trivial fiber bundle $p\colon M\to
  S^1$ and a $k \in \IQ$, $k > 0$  such that the induced map
  $p_*\colon \pi_1(M)\to \pi_1(S^1)=\IZ $ coincides with $k \cdot \phi$.
\end{definition}

\begin{theorem}\label{thm:dfl-fibered}
Let $M\ne S^1\times D^2$ be an irreducible 3-manifold. Then the following two statements hold:
\begin{enumerate}
\item[$(1)$] If $M$ is fibered, then for any \largehomoblank homomorphism
  $\mu\colon\pi_1(M)\to G$ to a residually finite group the pair $(M,\mu)$ is
  $L^2$-acyclic.
\item[$(2)$] If $\phi\in H^1(M;\IZ)=\operatorname{Hom}(\pi_1(M),\IZ)$ is a primitive
  fibered class, then there exists a $T\geq 1$ such that for any $\mathfrak{s}\in
  \spinc(M)$ and for any \largehomoblank homomorphism $\mu\colon\pi_1(M)\to G$ to a
  residually finite group the following inequalities hold
\[
\hspace{1cm}\begin{array}{rcll}
\rho^{(2)}(M,\mathfrak{s};\mu,\phi)(t)  &= &  \tmfrac{1}{2}\big(\phi(c_1(\mathfrak{s}))+  x_M(\phi)\big)  \cdot \ln(t)    
  \quad  &\text{for}\; t  < \frac{1}{T};
  \\[2mm]
\rho^{(2)}(M,\mathfrak{s};\mu,\phi)(t)  &=& 
  \tmfrac{1}{2}\big(\phi(c_1(\mathfrak{s}))- x_M(\phi)\big)  \cdot \ln(t)    
  \quad & \text{for}\; t > T.
  \end{array}
 \]
In fact one can choose $T$ to be the entropy of the monodromy.
\end{enumerate}
\end{theorem}

\begin{proof}
The first statement follows from \cite[Theorem~2.1]{Lueck(1994b)}.
Now we denote by $T$ the entropy of the monodromy of the primitive fibered class $\phi$.
By Theorem~8.5 of \cite{Dubois-Friedl-Lueck(2014Alexander)} there exists an $\mathfrak{s}\in \spinc(M)$ such that 
\begin{eqnarray*}
0 & = & \rho^{(2)}(M,\mathfrak{s};\mu,\phi)(t)  
\quad  \text{for}\; t  <\frac{1}{T};
\\
- x_M(\phi)  \cdot \ln(t)  
& = &
\rho^{(2)}(M,\mathfrak{s};\mu,\phi)(t)    
\quad  \text{for}\; t > T.
\end{eqnarray*}
The statement of the theorem follows from these inequalities in precisely the same way as
we concluded the proof of Theorem~\ref{the:lower_bound}.
\end{proof}

The next lemma improves on Theorem~\ref{thm:dfl-fibered} in so far as it gives us some
control over $\rho^{(2)}(M,\mathfrak{s};\mu,\phi)$ for all $t$. In particular the set of
$t$'s for which we have control does not depend on the choice of $\phi$.

\begin{lemma} \label{lem:mapping_tori_uniform_estimate} 
Let $(M,\mu\colon \pi_1(M)\to G)$ be an admissible pair and let $\mathfrak{s}\in \spinc(M)$. Then for any fibered $\phi\in H^1(M;\IQ)$  we have
\[
\begin{array}{rcll}
\rho^{(2)}(M,\mathfrak{s};\mu,\phi)(t)  &\le &  \tmfrac{1}{2}\big(\phi(c_1(\mathfrak{s}))+  x_M(\phi)\big)  \cdot \ln(t)    
  \quad  &\text{for}\; t  \le 1;
  \\[2mm]
\rho^{(2)}(M,\mathfrak{s};\mu,\phi)(t)  &\le& 
  \tmfrac{1}{2}\big(\phi(c_1(\mathfrak{s}))- x_M(\phi)\big)  \cdot \ln(t)    
  \quad & \text{for}\; t \ge 1.
  \end{array}
 \]
\end{lemma}

\begin{proof}
  Let $(M,\mu\colon \pi_1(M)\to G)$ be an admissible pair such that $M$ admits a fibered
  class. By Theorem~\ref{thm:dfl-fibered} the pair $(M,\mu)$ is $L^2$-acyclic. Let
  $\mathfrak{s}\in \spinc(M)$. The argument of the proof of
  Lemma~\ref{lem:primitiveisenough} shows that it suffices to prove the lemma for
  primitive fibered classes.  So let $\phi\in H^1(M;\IZ)=\operatorname{Hom}(\pi_1(M),\IZ)$
  be a primitive fibered class.

Consider a nested sequence of in $G$ normal subgroups
\[
G  \supseteq G_0 \supseteq G_1 \supseteq G_2 \supseteq \cdots 
\]
such that $G_i$ is contained in $\ker(G\to H_1(G;\IZ)_f)$, 
 the index $[\ker(G\to H_1(G;\IZ)_f) : G_i]$ 
is finite for $i \ge 0$
and the intersection $\bigcap_{i \ge 0} G_i$ is trivial. Put $Q_i := G/G_i$.
Denote  by $\pr_i \colon G \to Q_i$  the obvious projection.
Let $\mu_i \colon \pi_1(M) \to Q_i$ be the composition $\pr_i \circ \mu$. The homomorphisms $\mu_i$ are again \largehomo.

In the following we consider only the case where $M$ is closed, the case with boundary is
analogous. We apply Theorem~\ref{the:calculation_of_L2-torsion_from_a_presentation} (2) to
$M$. We denote the resulting square matrix over $\IZ[\pi]$ by $A$ and the resulting
elements in the group $\pi$ by $s,s'$. We write $A_i=\pr_i(A)$.  Define
\[ \eta(t)=\max\{0,|\phi(s)| \cdot \ln(t)\}+
\max\{0,|\phi(s')| \cdot \ln(t)\}.\]

As above, the pair $(M,\mu_i)$ is  $L^2$-acyclic. 
Our choice of $A$ and $s,s'$ ensures that
\begin{eqnarray*}
 \rho^{(2)}(M,\mathfrak{s};\mu,\phi) 
& = & 
\eta(t)- \ln\bigl({\det}_{\caln(G)}\bigl(\Lambda^G \circ \eta_{\phi^* \IC_t}(r_A)\bigr)\bigr);
\\
\rho^{(2)}(M,\mathfrak{s};\mu_i,\phi)
& = & \eta(t)
- \ln\bigl({\det}_{\caln(Q_i)}\bigl(\Lambda^{Q_i} \circ \eta_{\phi^* \IC_t}(r_{A_i})\bigr)\bigr).
\end{eqnarray*}
We conclude from Theorem~\ref{the:Twisted_Approximation_inequality}
\begin{eqnarray}
\quad \quad \quad \ln\bigl({\det}_{\caln(G)}\bigl(\Lambda^G \circ \eta_{\phi^* \IC_t}(r_A)\bigr)\bigr) 
& \ge   & 
\limsup_{i \to \infty} \ln\bigl({\det}_{\caln(Q_i)}\bigl(\Lambda^{Q_i} \circ \eta_{\phi^* \IC_t}(r_{A_i})\bigr)\bigr).
\label{lem:mapping_tori_uniform_estimate:inequality} 
\end{eqnarray}

By Theorem~\ref{thm:dfl-fibered}  there exists a $T\geq 1$ such that 
for any natural number $i$ we have
\[
\begin{array}{rcll}
\rho^{(2)}(M,\mathfrak{s};\mu_i,\phi)(t)  &= &  \tmfrac{1}{2}\big(\phi(c_1(\mathfrak{s}))+  x_M(\phi)\big)  \cdot \ln(t)    
  \quad  &\text{for}\; t  < \frac{1}{T};
  \\[2mm]
\rho^{(2)}(M,\mathfrak{s};\mu_i,\phi)(t)  &=& 
  \tmfrac{1}{2}\big(\phi(c_1(\mathfrak{s}))- x_M(\phi)\big)  \cdot \ln(t)    
  \quad & \text{for}\; t > T.
  \end{array}
 \]
This implies  
\[
\begin{array}{lcll}
\ln\bigl({\det}_{\caln(Q_i)}\bigl(\Lambda^G \circ \eta_{\phi^* \IC_t}(r_{A_i})\bigr)\bigr)   
& = & 
\eta(t)-\frac{1}{2}\big(\phi(c_1(\mathfrak{s}))+  x_M(\phi)\big)  \cdot \ln(t)     & \text{for} \; t < \frac{1}{T};
\\[2mm]
\ln\bigl({\det}_{\caln(Q_i)}\bigl(\Lambda^G \circ \eta_{\phi^* \IC_t}(r_{A_i})\bigr)\bigr)
& = & 
 \eta(t)-\frac{1}{2}\big(\phi(c_1(\mathfrak{s}))-  x_M(\phi)\big)  \cdot \ln(t)    & \text{for} \; t > T.
\end{array}
\]
Then Lemma~\ref{lem:uniform_estimate} applied to $\phi \colon Q_i \to \IZ$ yields 
\[
\begin{array}{lcll}
\ln\bigl({\det}_{\caln(Q_i)}\bigl(\Lambda^G \circ \eta_{\phi^* \IC_t}(r_{A_i})\bigr)\bigr)   
& \ge  & 
\eta(t)-\frac{1}{2}\big(\phi(c_1(\mathfrak{s}))+  x_M(\phi)\big)  \cdot \ln(t)    & \text{for} \; t \le 1
\\[2mm]
\ln\bigl({\det}_{\caln(Q_i)}\bigl(\Lambda^G \circ \eta_{\phi^* \IC_t}(r_{A_i})\bigr)\bigr)   
& \ge & 
  \eta(t)-\frac{1}{2}\big(\phi(c_1(\mathfrak{s}))- x_M(\phi)\big)  \cdot \ln(t)     & \text{for} \; t \ge  1.
\end{array}
\]
Since this holds for all $i \ge 0$ and all $t > 0$, we conclude 
from~\eqref{lem:mapping_tori_uniform_estimate:inequality} 
\[
\begin{array}{lcll}
\ln\bigl({\det}_{\caln(G)}\bigl(\Lambda^G \circ \eta_{\phi^* \IC_t}(r_A)\bigr)\bigr)   
& \ge  & 
\eta(t)-\frac{1}{2}\big(\phi(c_1(\mathfrak{s}))+  x_M(\phi)\big)  \cdot \ln(t)
& 
\text{for}\;  t \le 1;
\\[2mm]
\ln\bigl({\det}_{\caln(G)}\bigl(\Lambda^G \circ \eta_{\phi^* \IC_t}(r_A)\bigr)\bigr)   
& \ge & 
\eta(t)-\frac{1}{2}\big(\phi(c_1(\mathfrak{s}))-  x_M(\phi)\big)  \cdot \ln(t)
& 
\text{for}\;  t \ge 1.
\end{array}
\]
This implies 
\[
\begin{array}{lcll}
\rho^{(2)}(M,\mathfrak{s};\mu,\phi)   
& \le & \frac{1}{2}\big(\phi(c_1(\mathfrak{s}))+  x_M(\phi)\big)  \cdot \ln(t) &
\quad  \text{for}\; t \le 1;
\\[2mm]
\rho^{(2)}(M,\mathfrak{s};\mu,\phi) 
& \le & 
\frac{1}{2}\big(\phi(c_1(\mathfrak{s}))-  x_M(\phi)\big)  \cdot \ln(t) & \quad  \text{for}\; t \ge 1.
\end{array}
\]
\end{proof}

\begin{lemma}\label{lem:continuity_of_det}
  Let $\Gamma $ be a group that is virtually finitely generated free abelian. Consider a
  finite subset $S \subseteq \Gamma$.  Then for any natural number $n$ the function
  \begin{multline*}
  \{A \in M_{n,n}(\IC \Gamma) \mid \supp_{\Gamma}(A) \subseteq S\} \to [0,\infty], 
  \\
   A  \mapsto \begin{cases} {\det}_{\caln(\Gamma)}(\Lambda^{\Gamma}(r_A)) 
   & 
   \text{if} \; \Lambda^{\Gamma}(r_A) \;
    \text{is a weak isomorphism;} 
    \\ 
    0 
    & 
    \text{otherwise,}
  \end{cases}
  \end{multline*}
is continuous with respect to the standard topology on the source coming from 
  the structure of a finite-dimensional complex vector space.
\end{lemma}

\begin{proof}
  Let $i \colon \IZ^d \to \Gamma$ be an inclusion whose image has finite index in
  $\Gamma$.  Fix a map of sets $\sigma \colon \im(i)\backslash \Gamma \to \Gamma$ whose
  composition with the projection $\Gamma \to \im(i)\backslash \Gamma$ is the identity. Put
  $m = [\Gamma: \im(i)]$.  With this choice the finitely generated free
  $\IC[\IZ^d]$-module $i^* \IC \Gamma$ obtained from $\IC \Gamma$ by restriction with $i$
  inherits a preferred $\IC[\IZ^d]$-basis.  Hence there is a finite subset 
$T \subseteq   \IZ^d$ and a $\IC$-linear (and hence continuous) map
\[
i^* \colon \{A \in M_{n,n}(\IC \Gamma) \mid \supp_{\Gamma }(A) \subseteq S\} \to 
\{B \in M_{mn,mn}(\IC[\IZ^d]) \mid \supp_{\IZ^d}(B) \subseteq T\} 
\]
such that $i^*\Lambda^{\Gamma }(r_A) = \Lambda^{\IZ^d}(r_{i^*A})$. Since
\[
{\det}_{\caln(\IZ^d)}(i^*\Lambda^{\Gamma }(r_A)) = m \cdot {\det}_{\caln(\Gamma )}(r_A))
\]
holds for any $A \in M_{n,n}(\IC \Gamma)$ by~\cite[Theorem~3.14~(5) on page~128]{Lueck(2002)},
it suffices to prove the claim in the special case $\Gamma = \IZ^d$.

As $\det_{\IC[\IZ^d]} \colon M_{n,n}(\IC[\IZ^d]) \to  M_{1,1}(\IC[\IZ^d])$ is continuous
and for $A \in M_{n,n}(\IC \IZ^d)$ with $\supp_{\IZ^d}(A) \subset S$ we have
$\supp_{\IZ^d}(\det_{\IC[\IZ^d]}) \subseteq S^n$ for $S^n = \{g_1 \cdot g_2 \cdot \,\cdots \, \cdot g_n \mid g_i \in S\}$, 
we conclude from~\cite[Lemma~6.25]{Lueck(2015twisting)} 
that it suffices to treat the case $n = 1$.
Since the Mahler measure of a non-trivial element $p \in \IC[\IZ^d]$ is equal to
$\det_{\IC[\IZ^d]}\bigl(\Lambda^{\IZ^d}(r_p) \colon L^2(\IZ^d) \to L^2(\IZ^d)\bigr)$
and defined  to be zero for $p = 0$, Lemma~\ref{lem:continuity_of_det} 
follows from a continuity theorem for Mahler measures proved by  Boyd~\cite[p.~127]{Boyd(1998approximation)}.
\end{proof}

\begin{definition}[Quasi-fibered classes]
  \label{def:fibered_and_quasi_fibered}
  Let $N$ be a $3$-manifold.  We call an
  element $\phi \in H^1(N;\IR)$ \emph{quasi-fibered}, if there exists a sequence of
  fibered elements $\phi_n \in H^1(N;\IQ)$ converging to $\phi$ in $H^1(N;\IR)$.
\end{definition}

Notice that obviously any fibered $\phi$ is non-trivial.  The next theorem generalizes the
inequalities of Lemma~\ref{lem:mapping_tori_uniform_estimate} for fibered classes to
quasi-fibered classes. This theorem can be viewed as the key technical result of this
paper.

\begin{theorem}[Upper bound  in the quasi-fibered case]
\label{the:upper_bound_in_the_quasi_fibered_case}
Let $(M,\mu)$ be an admissible pair, $\mathfrak{s}\in  \spinc(M)$ and let $\phi \in H^1(M;\IR)$ be a quasi-fibered class. Then
\[
\begin{array}{rcll}
\rho^{(2)}(M,\mathfrak{s};\mu,\phi)(t)  &\le &  \tmfrac{1}{2}\big(\phi(c_1(\mathfrak{s}))+  x_M(\phi)\big)  \cdot \ln(t)    
  \quad  &\text{for}\; t  \le 1;
  \\[2mm]
\rho^{(2)}(M,\mathfrak{s};\mu,\phi)(t)  &\le& 
  \tmfrac{1}{2}\big(\phi(c_1(\mathfrak{s}))- x_M(\phi )\big)  \cdot \ln(t)    
  \quad & \text{for}\; t \ge 1.
  \end{array}
 \]
\end{theorem}

\begin{proof} 
We only treat the case, where $\partial M$ is empty, the other case is completely analogous: in the proof below  one needs to replace
Theorem~\ref{the:calculation_of_L2-torsion_from_a_presentation} (2) by
Theorem~\ref{the:calculation_of_L2-torsion_from_a_presentation} (1).
We write $\pi=\pi_1(M)$ and we pick $\mathfrak{s}\in  \spinc(M)$. 

First recall that our assumption that $\mu\colon \pi\to G$ is \largehomoblank implies that
the projection $\pi\to H_1(M)_f$ factors through $\mu$ and a map $\nu\colon G\to
H_1(M)_f$. Since $G$ is residually finite we can choose a sequence of normal subgroups of
$G$
\[
G  \supseteq G_0 \supseteq G_1 \supseteq G_2 \supseteq \cdots 
\]
such that $G_i$ is contained in $\ker(\nu\colon G\to H_1(M)_f)$, the index $[\ker(\nu) :
G_i]$ is finite for $i \ge 0$ and the intersection $\bigcap_{i \ge 0} G_i$ is trivial. Put
$Q_i := G/G_i$.  Denote by $\mu_i \colon \pi \to Q_i $ the composition of the projection
$\pr_i \colon G \to Q_i$ with $\mu$. Note that $\mu_i$ is again a \largehomoblank
homomorphism. Recall that this implies in particular that we can make the identifications
\[ H^1(M;\IR)=\operatorname{Hom}(H_1(\pi)_f,\IR)=
\operatorname{Hom}(\pi,\IR)=\operatorname{Hom}(G,\IR)=\operatorname{Hom}(G_i,\IR).\]

We apply Theorem~\ref{the:calculation_of_L2-torsion_from_a_presentation}~(2) to $M$ and
$\mathfrak{s}$.  We denote the resulting square matrix over $\IZ[\pi]$ by $A$ and the
resulting elements in $\pi$ by $s,s'$. For each $i\in \IN$ we write $A_i=\pr_i(A)$.
Define for any homomorphism $\psi \colon H_1(M)_f \to \IR$
\begin{eqnarray*}
\xi(\psi)(t)  
& = & \max\big\{0,\bigl(|\psi \circ \nu \circ \mu(s)| + |\psi \circ \nu \circ \mu(s')|\bigr)\cdot \ln(t)\big\}.
\end{eqnarray*}

We start out with the following claim.

\begin{claim}
For each $i\in \IN$ we have the inequalities
\[
\begin{array}{lcll}
\rho^{(2)}(M,\mathfrak{s};\mu_i,\phi_i)(t) & \le&   \frac{1}{2}\big(\phi(c_1(\mathfrak{s}))+  x_M(\phi)\big)  \cdot \ln(t)    
  \quad & \text{for}\; t  \le 1;\\[2mm]
\rho^{(2)}(M,\mathfrak{s};\mu_i,\phi_i)(t) & \le &
  \frac{1}{2}\big(\phi(c_1(\mathfrak{s}))- x_M(\phi)\big)  \cdot \ln(t)    
  \quad & \text{for}\; t \ge 1.
\end{array}
\]
\end{claim}

Let $i\in\IN$. 
Since $\phi \in H^1(M;\IR)$ is quasi-fibered, there exists a sequence of fibered elements
$\phi_n \in H^1(M;\IQ)$ converging to $\phi$. 
By
Lemma~\ref{lem:mapping_tori_uniform_estimate}
we know that for each $i$ and $n$  we have 
\begin{eqnarray}\label{Lahm_I}
\rho^{(2)}(M,\mathfrak{s};\mu_i,\phi_n)(t) & \le  & \tmfrac{1}{2}\big(\phi_n(c_1(\mathfrak{s}))+  x_M(\phi_n)\big)  \cdot \ln(t)    
  \quad  \text{for}\; t  \le 1;\\
\label{Lahm_II}
\rho^{(2)}(M,\mathfrak{s};\mu_i,\phi_n)(t)&  \le &
  \tmfrac{1}{2}\big(\phi_n(c_1(\mathfrak{s}))- x_M(\phi_n)\big)  \cdot \ln(t)    
  \quad  \text{for}\; t \ge 1.
\end{eqnarray}
By  Theorem~\ref{the:calculation_of_L2-torsion_from_a_presentation}~(2)  we have
\begin{eqnarray}
\label{Boateng_I}
\hspace{10mm}  \rho^{(2)}(M,\mathfrak{s};\mu_i,\phi_n)(t)
& = & 
\xi(\phi_n)(t)-\ln \bigl({\det}_{\caln(Q_i)}(\Lambda^{Q_i} \circ \eta_{\phi_n^* \IC_t}(r_{A_i}))\bigr);
\\
\label{Boateng_II}
\rho^{(2)}(M,\mathfrak{s};\mu_i,\phi)(t)
& = & 
\xi(\phi)(t)-\ln \bigl({\det}_{\caln(Q_i)}(\Lambda^{Q_i} \circ \eta_{\phi^* \IC_t}(r_{A_i}))\bigr).
\end{eqnarray}

Since $\phi_n$ converges to $\phi$ and the kernel of the projection map $Q_i  \to H_1(M)_f$ is finite, we get  
from Lemma~\ref{lem:continuity_of_det} that
\[
 \lim_{n \to \infty}\ln \bigl({\det}_{\caln(Q_i)}(\Lambda^{Q_i} \circ \eta_{\phi_n^* \IC_t}(r_{A_i}) \bigr)\bigr) 
 = 
\ln \bigl({\det}_{\caln(Q_i)}(\Lambda^{Q_i} \circ \eta_{\phi^* \IC_t}(r_{A_i})\bigr)\bigr).\]
This equality, together with Equations~\eqref{Boateng_I} and~\eqref{Boateng_II}  and the observation 
that for any $t\in (0,\infty)$ the equality  $ \lim_{n \to \infty} \xi(\phi_n)(t)=\xi(\phi)(t)$ holds, implies that 
\begin{eqnarray*}
\rho^{(2)}(M,\mathfrak{s};\mu_i,\phi)(t)&=&
\lim_{n \to \infty} \rho^{(2)}(M,\mathfrak{s};\mu_i,\phi_n)(t) 
  \quad \text{for all }t\in (0,\infty).
\end{eqnarray*}
The desired inequalities for $\rho^{(2)}(M,\mathfrak{s};\mu_i,\phi)(t)$ now follow from (\ref{Lahm_I}) and (\ref{Lahm_II}).
This concludes the proof of the claim.

Now the theorem  follows from the claim we just proved and the following claim.

\begin{claim}
For each $t\in (0,\infty)$ we have 
\[
\rho^{(2)}(M,\mathfrak{s},\mu,\phi)(t)
 \le 
\liminf_{i \to \infty} \rho^{(2)}(M,\mathfrak{s},\mu_i,\phi_i)(t).\]
\end{claim}

This claim is proved as follows. By Theorem~\ref{thm:dfl-fibered} 
we know that  the pairs $(M,\mu)$ and $(M,\mu_i)$ are  $L^2$-acyclic.
By  Theorem~\ref{the:calculation_of_L2-torsion_from_a_presentation}~(2) we have
\begin{eqnarray}
\label{Boateng_III}
\rho^{(2)}(M,\mathfrak{s},\mu,\phi)(t)
& = & 
\xi(\phi)(t)-\ln \bigl({\det}_{\caln(G)}(\Lambda^G \circ \eta_{\phi^* \IC_t}(r_A))\bigr).
\end{eqnarray}
Recall that the kernel of $Q_i \to H_1(M)_f$ is finite und
that $Q_i \to H_1(M)_f$ is surjective.
Now we apply Theorem~\ref{the:Twisted_Approximation_inequality}
to $\phi  \colon G \to \IR$. For all $t\in (0,\infty)$ we obtain
\begin{eqnarray*}
\ln \bigl({\det}_{\caln(G)}(\Lambda^G \circ \eta_{\phi^* \IC_t}(r_{A}))\bigr)  
&  \ge & 
\limsup_{i \to \infty} \ln \bigl({\det}_{\caln(Q_i)}(\Lambda^{Q_i} \circ \eta_{\phi_i^* \IC_t}(r_{A_i}))\bigr).
\end{eqnarray*}
Now apply~\eqref{Boateng_II} and~\eqref{Boateng_III}. This finishes the proof
of Theorem~\ref{the:upper_bound_in_the_quasi_fibered_case}.
\end{proof}

For convenience we also state the result which follows from combining
Theorem~\ref{the:lower_bound} with
Theorem~\ref{the:upper_bound_in_the_quasi_fibered_case}.

\begin{theorem}[Lower and upper bounds combined in the quasi-fibered case]
  \label{the:Lower_and_upper_bounds_in_the_quasi-fibered_case}
  Let $M\ne S^1\times D^2$ be an irreducible $3$-manifold with infinite fundamental group
  $\pi$.  Let $\phi \in H^1(M;\IQ)$ be a quasi-fibered class.


Then there exists a $D\in \IR$ such that for any $\mathfrak{s}\in \spinc(M)$ and any
\largehomoblank homomorphism $\mu\colon \pi_1(M)\to G$, where $G$ is residually finite and
countable, the pair $(M,\mu)$ is $L^2$-acyclic and such that for $t\leq 1$
\begin{align*}
\begin{array}{crcccl}
\tmfrac{1}{2}\big(\phi(c_1(\mathfrak{s}))+  x_M(\phi)\big)   \ln t  - D  
\hspace{-0.2cm}& \le\hspace{-0.2cm} & 
\rho^{(2)}(M,\mathfrak{s};\mu,\phi)(t)\hspace{-0.2cm}& \le \hspace{-0.2cm}&  \tmfrac{1}{2}\big(\phi(c_1(\mathfrak{s}))+  x_M(\phi)\big)  \ln t     
\intertext{and such that for $t\ge 1$}
  \tmfrac{1}{2}\big(\phi(c_1(\mathfrak{s}))-  x_M(\phi)\big)   \ln t  - D \hspace{-0.2cm} & \le \hspace{-0.2cm} &
\rho^{(2)}(M,\mathfrak{s};\mu,\phi)(t)\hspace{-0.2cm} & \le \hspace{-0.2cm} &
  \tmfrac{1}{2}\big(\phi(c_1(\mathfrak{s}))- x_M(\phi )\big)  \ln t.
  \end{array}
 \end{align*}
 In particular we get
 \[
 \deg\bigl({\rho}(M,\mathfrak{s};\mu, \phi)\bigr) = - x_M(\phi).
 \]
\end{theorem}


 \typeout{-----------   Section 5: Lower and upper bounds  combined -----------------------------------}

\section{Proof of the main theorem}
The following is the main theorem of this paper.

\begin{theorem}[Main theorem]
\label{the:Lower_and_upper_bounds}
Let $M$ be an irreducible $3$-manifold  with
infinite fundamental group $\pi$
which is not a closed graph manifold and not homeomorphic to $S^1 \times D^2$. 
Let $\mathfrak{s}\in  \spinc(M)$ and write $\pi=\pi_1(M)$.

Then there exists a \largehomoblank epimorphism $\alpha\colon \pi\to \Gamma$ to a
virtually finitely generated free abelian group such that the following holds: For any
$\phi \in H^1(M;\IQ)$ and any factorization of $\alpha \colon \pi \to \Gamma$ into group
homomorphisms $\pi \xrightarrow{\mu} G \xrightarrow{\nu} \Gamma$ for a residually finite
countable group $G$, there exists a real number $D$ depending only on $\phi$ but not on
$\mu$ such that for $t\leq 1$
\begin{align*}
\begin{array}{crcccr}
\tmfrac{1}{2}\big(\phi(c_1(\mathfrak{s}))+  x_M(\phi)\big)   \ln t  - D  
\hspace{-0.2cm}& \le\hspace{-0.2cm} & 
\rho^{(2)}(M,\mathfrak{s};\mu,\phi)(t)\hspace{-0.2cm}& \le \hspace{-0.2cm}&  
\tmfrac{1}{2}\big(\phi(c_1(\mathfrak{s}))+  x_M(\phi)\big)  \ln t     
\intertext{and such that for $t\ge 1$}
  \tmfrac{1}{2}\big(\phi(c_1(\mathfrak{s}))-  x_M(\phi)\big)   \ln t  - D \hspace{-0.2cm} & \le \hspace{-0.2cm} &
\rho^{(2)}(M,\mathfrak{s};\mu,\phi)(t)\hspace{-0.2cm} & \le \hspace{-0.2cm} &
  \tmfrac{1}{2}\big(\phi(c_1(\mathfrak{s}))- x_M(\phi )\big)  \ln t.
  \end{array}
 \end{align*}
 In particular we get
 \[
 \deg\bigl({\rho}(M,\mathfrak{s};\mu, \phi)\bigr) = - x_M(\phi).
 \]
\end{theorem}

\begin{proof}
  As explained in~\cite[Section~10]{Dubois-Friedl-Lueck(2014Alexander)}, we conclude from
  combining~\cite{Agol(2008),Agol(2013),Liu(2013),Przytycki-Wise(2012),Przytycki-Wise(2014),Wise(2012raggs),Wise(2012hierachy)}
  that there exists a finite regular covering $p \colon N \to M$ such that for any $\phi
  \in H^1(M;\IR)$ its pullback $p^*\phi \in H^1(N;\IR)$ is quasi-fibered.  Let $k$ be the
  number of sheets of $p$.  Let $\pr_N \colon \pi_1(N) \to H_1(N)_f$ and $\pr_M \colon
  \pi_1(M) \to H_1(M)_f$ be the canonical projections.  The kernel of $\pr_N$ is a
  characteristic subgroup of $\pi_1(N)$. The regular finite covering $p$ induces an
  injection $\pi_1(p) \colon \pi_1(N) \to \pi_1(M)$ such that the image of $\pi_1(p)$ is a
  normal subgroup of $\pi_1(M)$ of finite index.  Let $\Gamma$ be the quotient of
  $\pi_1(M)$ by the normal subgroup $\pi_1(p)(\ker(\pr_N))$.  Let $\alpha \colon \pi_1(M)
  \to \Gamma$ be the projection.  Because of $H_1(p;\IZ)_f \circ \pr_N = \pr_M \circ \,
  \pi_1(p)$ we know that $\pi_1(p)(\ker(\pr_N))$ is contained in the kernel of the
  canonical projection $\pr_M \colon \pi_1(M)\to H_1(M)_f$. This implies that
  $\alpha\colon \pi_1(M)\to \Gamma$ is \largehomo, which means in particular that there
  exists precisely one epimorphism $\beta \colon \Gamma \to H_1(M)_f$ satisfying $\pr_M =
  \beta \circ \alpha$.  Moreover, $\alpha \circ \pi_1(p)$ factorizes over $\pr_N$ into an
  injective homomorphism $j \colon H_1(N)_f \to \Gamma$ with finite cokernel.  Hence
  $\Gamma$ is virtually finitely generated free abelian.

 Consider any factorization of the homomorphism $\alpha \colon \pi_1(M) \to \Gamma$
into group homomorphisms 
$\pi_1(M) \xrightarrow{\mu} G \xrightarrow{\nu} \Gamma$ 
with  residually finite countable $G$.

Let $G' $ be the quotient of
$\pi_1(N)$ by the normal subgroup $\pi_1(p)^{-1}(\ker(\mu))$ and
$\mu' \colon \pi_1(N) \to G'$ be the projection. Since the kernel of
$\mu'$ and of $\mu \circ \pi_1(p)$ agree, there is precisely
one injective group homomorphism  $i \colon G' \to G$ satisfying 
$\mu \circ \pi_1(p) = i \circ \mu'$.  The kernel of $\mu'$ is contained in the kernel of 
$\pr_N \colon \pi_1(N) \to H_1(N)_f$ since $j$ is injective and we have
$j \circ \pr_N = \nu \circ i \circ \mu'$. Hence there is precisely one group homomorphism
$\nu' \colon G' \to H_1(N)_f$ satisfying $\nu' \circ \mu' = \pr_N$. In particular $\mu'$ is a 
\largehomoblank homomorphism. One easily checks that 
the following diagram commutes, and all vertical arrows are injective, the indices 
$[\pi_1(N): \im(\pi_1(p)]$ and $[\Gamma: H_1(N)_f]$ are finite,
and $\mu'$, $\nu'$ and $\beta$ are surjective.
\[
\xymatrix{\pi_1(N) \ar[r]^-{\mu'} \ar[d]^{\pi_1(p)} \ar@/^{5mm}/[rr]^{\pr_N}
& 
G' \ar[r]^-{\nu'} \ar[d]^i
& 
H_1(N)_f \ar[rd]^{H_1(p)_f} \ar[d]^j
\\
\pi_1(M) \ar[r]^-{\mu} \ar@/_{5mm}/[rr]_-{\alpha} \ar@/_{10mm}/[rrr]_{\pr_M}
& 
G \ar[r]^{\nu}
&
\Gamma \ar[r]^-{\beta}
& 
H_1(M)_f 
}
\]
Since $G$ is residually finite and countable, the group $G'$ is residually finite and countable. 

Now let $\mathfrak{s}\in \spinc(M)$ and let $\phi\in
H^1(M;\IQ)=\operatorname{Hom}(H_1(M)_f;\IQ)$.  We write $\mathfrak{s}'=p^*(\mathfrak{s})$
  and $\phi'=p^*(\phi)$. Furthermore we put $c=c_1(\mathfrak{s})$ and
$c'=c_1(\mathfrak{s}')$.  Since $\phi'\in H^1(N;\IQ)=\operatorname{Hom}(H_1(N)_f;\IQ)$ is
quasi-fibered we can appeal to
Theorem~\ref{the:Lower_and_upper_bounds_in_the_quasi-fibered_case}.  In our context it
says that $(N,\mu')$ is $L^2$-acyclic and that there exists a real number $D'$ depending
only on $\phi'$ but not on $\mu'$ such that for $t\leq 1$
\begin{align*}
\begin{array}{crcccr}
\hspace{0.4cm}
\tmfrac{1}{2}\big(\phi'(c')+  x_N(\phi')\big)   \ln t  - D'  
\hspace{-0.2cm}& \le\hspace{-0.2cm} & 
\rho^{(2)}(N,\mathfrak{s}';\mu',\phi')(t)\hspace{-0.2cm}& \le \hspace{-0.2cm}&  \tmfrac{1}{2}\big(\phi'(c')+  x_N(\phi')\big)  \ln t     
\intertext{and such that for $t\ge 1$}
  \tmfrac{1}{2}\big(\phi'(c')-  x_N(\phi')\big)   \ln t  - D' \hspace{-0.2cm} & \le \hspace{-0.2cm} &
\rho^{(2)}(N,\mathfrak{s}';\mu,\phi')(t)\hspace{-0.2cm} & \le \hspace{-0.2cm} &
  \tmfrac{1}{2}\big(\phi'(c')- x_N(\phi' )\big)  \ln t.
  \end{array}
 \end{align*}
We now set $D:=\frac{1}{k}D'$. 
The theorem now follows from these inequalities and the following equalities
\[\begin{array}{rcl}
x_M(\phi)&=&\frac{1}{k}x_N(\phi')\\[2mm]
\rho^{(2)}(M,\mathfrak{s};\mu,\phi)(t)&=&\frac{1}{k}\rho^{(2)}(N,\mathfrak{s}';\mu',\phi')(t)\quad \mbox{for all }t\\[2mm]
\phi(c_1(\mathfrak{s}))&=&\frac{1}{k}\phi'(c_1(\mathfrak{s}')).
\end{array}
\]
Here the first equality is (\ref{finite_coverings_Thurston_norm}) and the second one 
Theorem~\ref{the:Properties_of_the_twisted_L2-torsion_function} (3).
The third one follows easily from the definitions.
\end{proof}

\begin{remark}[Graph manifolds]\label{rem:graph_manifolds}
  The proof of Theorem~\ref{the:Lower_and_upper_bounds_in_the_quasi-fibered_case} does not
  cover closed graph manifolds. However, for a graph manifold $M$ together with a
  \largehomoblank homomorphism $\mu\colon \pi_1(M)\to G$, for which $(M,\mu)$ is
  $L^2$-acyclic, together with a class $\phi \in H^1(M;\IR)$ the $L^2$-torsion function
  has been computed explicitly in~\cite[Theorem~8.2]{Dubois-Friedl-Lueck(2014Alexander)} and in~\cite{Herrmann(2015)}
  to be
\[
\overline{\rho}^{(2)}(M;\mu,\phi)(t)  \doteq \min\{0,- x_M(\phi) \cdot \ln(t)\},
\]
provided that the image of the regular fiber under $\mu$ 
is an element of infinite order and $M$ is neither  $S^1 \times D^2$ nor $S^1 \times S^2$.
This implies
\[
\deg\bigl(\overline{\rho}^{(2)}(M;\mu,\phi)\bigr)   = - x_M(\phi).
\]
\end{remark}

\begin{remark}[The role of $\Gamma$] \label{rem:role_of_Gamma} In
  Theorem~\ref{the:Lower_and_upper_bounds} the group $\Gamma$ is in some sense optimal.
  Namely, one cannot expect $\Gamma = H_1(M)_f$ and $\beta = \id_{\Gamma}$ in
  Theorem~\ref{the:Lower_and_upper_bounds}. For instance, let $K \subseteq S^3$ be a
  non-trivial knot. Take $M$ to be the 
  $3$-manifold given by the complement of an open tubular
  neighborhood of the knot. Then $\deg(\overline{\rho}(M;\mu,\phi))$ for 
  $\mu \colon   \pi_1(M) \to H_1(M)_f$ the canonical projection and 
  $\phi \colon H_1(M)_f   \xrightarrow{\cong} \IZ$ an isomorphism is 
  just the degree of the Alexander polynomial
  of the knot $K$ which is not the Thurston norm $x_M(\phi)$ in general,
  see~\cite[Section~7.3]{Dubois-Friedl-Lueck(2014Alexander)}.
\end{remark}

\begin{example}[$S^1\times D^2$ and $S^1 \times S^2$]
\label{exa:S1_timesD2_and_S1_times_S2}
Consider a homomorphism $\phi \colon H_1(S^1 \times D^2) \xrightarrow{\cong} \IZ$.
Let $k$ be the index $[\IZ : \im(\phi)]$ if $\phi$ is non-trivial, and let $k = 0$ if $\phi$ is trivial.
Then we conclude from
Theorem~\ref{the:Properties_of_the_twisted_L2-torsion_function}~(4),~\eqref{x_for_fiber_bundles},  
and~\cite[Theorem~7.10]{Lueck(2015twisting)}
\begin{eqnarray*}
x_{S^1 \times D^2}(\phi) & = & 0;
\\
\deg\big(\overline{\rho}(\widetilde{S^1 \times D^2};\phi)\big) 
& = & k.
\end{eqnarray*}
Hence we have to exclude $S^1 \times D^2$ in Theorem~\ref{the:Lower_and_upper_bounds}.
Analogously we get 
\begin{eqnarray*}
x_{S^1 \times S^2}(\phi) & = & 0;
\\
\deg\big(\overline{\rho}(\widetilde{S^1 \times S^2};\phi)\big)
& = & 2 \cdot k,
\end{eqnarray*}
so that we cannot replace ``irreducible'' by ``prime'' in Theorem~\ref{the:Lower_and_upper_bounds}.
\end{example}

We conclude the paper with the proof of  Theorem~\ref{the:main_result_introduction}.

\begin{proof}[Proof of Theorem~\ref{the:main_result_introduction}]
  Let $M$ be an irreducible 3-manifold with infinite fundamental group. If $M$ is a graph
  manifold, then the statement is proved in Remark~\ref{rem:graph_manifolds}. Now suppose
  that $M$ is not a graph manifold. In this case the theorem follows from
  Theorem~\ref{the:Lower_and_upper_bounds}, applied in the special case $G = \pi_1(M)$,
  $\mu = \id_{\pi_1(M)}$ and $\nu = \alpha$. Here we use that by work of
  Hempel~\cite{Hempel(1987)} and the proof of the Geometrization Conjecture fundamental
  groups of 3-manifolds are residually finite.
\end{proof}

\typeout{-------------------------------------- References  ---------------------r------------------}



\end{document}